\documentclass[article,onefignum,onetabnum, final]{siamart171218}
\pdfoutput=1

\usepackage{lipsum}
\usepackage{amsfonts}
\usepackage{graphicx}
\usepackage{epstopdf}
\usepackage{algorithmic}
\usepackage{verbatim}
\usepackage{tikz}
\usepackage{pdfcomment}
\usepackage{pgfplots}
\usepackage{subfig}
\usepackage{mathtools}
\usepackage{bigints}
\usepackage{tabularx,ragged2e,booktabs,caption}
\ifpdf
  \DeclareGraphicsExtensions{.eps,.pdf,.png,.jpg}
\else
  \DeclareGraphicsExtensions{.eps}
\fi
\renewenvironment{equation*}{\[}{\]\ignorespacesafterend}

\newcommand{\eps}{\varepsilon}
\newsiamremark{remark}{Remark}
\newsiamremark{hypothesis}{Hypothesis}
\crefname{hypothesis}{Hypothesis}{Hypotheses}
\newsiamthm{claim}{Claim}
\newsiamremark{assumption}{Assumption}

\headers{A priori error estimates for Westervelt's wave equation}{V. Nikoli\' c and B. Wohlmuth}

\title{A priori error estimates for the finite element approximation of Westervelt's quasilinear acoustic wave equation \\[2mm]}
\author{Vanja Nikoli\' c\footnotemark[2]
\and Barbara Wohlmuth\footnotemark[2]}

\usepackage{amsopn}

\ifpdf
\hypersetup{
  pdftitle={A priori error estimates for the Westervelt equation},
  pdfauthor={V. Nikoli\'c and B. Wohlmuth}
}
\fi

\begin{document}
\maketitle
\vspace{4mm}
\begin{abstract}
We study the spatial discretization of Westervelt's quasilinear strongly damped wave equation by piecewise linear finite elements. Our approach employs the Banach fixed-point theorem combined with a priori analysis of a linear wave model with variable coefficients.  Degeneracy of the semi-discrete Westervelt equation is avoided by relying on the inverse estimates for finite element functions and the stability and approximation properties of the interpolation operator. In this way, we obtain optimal convergence rates in $L^2$-based spatial norms for sufficiently small data and mesh size and an appropriate choice of initial approximations. Numerical experiments in a setting of a 1D channel as well as for a focused-ultrasound problem illustrate our theoretical findings.
\end{abstract}
~\\
\begin{keywords}
finite element method, a priori analysis, nonlinear acoustics, Westervelt's equation
\end{keywords}

\begin{AMS}
  35L05, 65M15, 65M60
\end{AMS}
\let\thefootnote\relax\footnote{$^\dagger$Technical University of Munich, Department of Mathematics, Chair of Numerical Mathematics, Boltzmannstra\ss e 3, 85748 Garching, Germany 
  (\email{vanja.nikolic@ma.tum.de}, \email{wohlmuth@ma.tum.de}).}
\section{Introduction}
The goal of the present work is to analyze a spatial discretization by piecewise linear finite elements in nonlinear acoustics. To this end, we study a discretization of Westervelt's wave equation for the acoustic pressure $u$
\begin{align} \label{Westervelt}
(1-2ku)u_{tt}-c^2 \Delta u-b \Delta u_t=2k u_t^2,
\end{align}
which represents a classical model for nonlinear ultrasound propagation through thermoviscous fluids~\cite{westervelt1963parametric}. Our research is motivated by a rising number of nonlinear ultrasound applications in medicine and industry~\cite{bjorno2002forty, fierro2015nonlinear, muller2008nonlinear, novell2009exploitation, pinton2011effects}. In~\eqref{Westervelt}, the constant $c$ denotes the speed of sound, $b$ is the sound diffusivity, and $k=\beta_a/(\varrho c^2)$, where $\varrho$ is the mass density and $\beta_a$ the coefficient of nonlinearity of the medium.  \\
\indent Westervelt's equation is a strongly damped quasilinear wave equation with potential degeneracy due to the factor $1-2ku$ next to the second time derivative. For its derivation and the theoretical foundations of nonlinear acoustics, we refer to~\cite{crighton1979model, enflo2006theory,hamilton1998nonlinear, westervelt1963parametric}, while results on the existence of smooth solutions of \eqref{Westervelt} can be found in~\cite{KaltenbacherLasiecka_Westervelt, KaltenbacherLasieckaVeljovic, meyer2011optimal}. Efficient simulation of the Westervelt equation and, in general, nonlinear sound propagation by the finite element method has been an active area of research. We refer to, e.g.,~\cite{fritz2018well, hoffelner2001finite, kagawa1992finite, manfred, muhr2018self, muhr2017isogeometric, tsuchiya2003finite, walsh2007finite}, which all focus on algorithmic aspects of finite element discretizations without any a priori analysis. \\
\indent Error analysis for the standard finite element discretization of linear wave equations is an extensively studied topic; see, e.g.,~\cite{baker1976error, bales1994continuous, bangerth1999finite, dupont1973, georgoulis2013posteriori, larsson1991finite, thomee2004maximum} and the references given therein. In particular, we single out the work on a priori analysis in~\cite{baker1976error} which provides $L^\infty(0,T; L^2)$ error estimates for the undamped linear wave equation and the results on error bounds for strongly damped linear wave equations~\cite{larsson1991finite, thomee2004maximum}. Results on a class of nonlinear wave equations of a divergent type are also well-established. In~\cite{dendy1977galerkin}, error analysis is provided for a semi-discretization of nonlinear wave equations of the form \[ u_{tt} - \sum _{i = 1}^n \frac{\partial }{\partial x_i } A_i (x, \nabla u) = f(x,t, u, \nabla u),\]
with a monotonicity condition on the corresponding bilinear form; cf.~\cite[Theorem 3.2]{dendy1977galerkin}.
In~\cite{suli2000priori}, semi-discretization for the following damped model is considered
\begin{equation*}
    \begin{aligned}
        u_{tt}-\Delta b(u)+\frac{\partial}{\partial t}(a(u))=f(x,t),
    \end{aligned}
\end{equation*}
with $b'(u) \geq M_0 >0$.  In~\cite{emmrich2015full}, convergence of a full discretization for a class of nonlinear second-order in time evolution equations is provided where the operator acting on the first time derivative is assumed to be hemicontinuous, monotone, coercive, and to fulfill a certain growth condition. Moreover, the operator acting on the solution is assumed to be linear, bounded, symmetric, and strongly positive. We also mention the results in~\cite{makridakis1993finite} on a class of problems of nonlinear elastodynamic and in~\cite{ortner2007discontinuous} on the discontinuous Galerkin methods for a class of divergent-type nonlinear hyperbolic equations. \\
\indent This work contributes to the finite element analysis of Westervelt's equation in two ways. We first prove that, coupled with non-zero initial conditions and homogeneous Dirichlet data, its semi-discretization by piecewise linear finite elements has a unique solution which remains bounded in an appropriately chosen norm. Secondly, we derive an optimal a priori error estimate that has the form
\begin{equation*} 
\begin{multlined}[c] \|u-u_h\|_{L^\infty(0,T; L^2(\Omega))}+\left\|u_t-u_{h,t} \right\|_{L^
\infty(0,T; L^2(\Omega))}+\left\|u_{tt}-u_{h,tt} \right\|_{L^
2(0,T; L^2(\Omega))}\\[2mm]
+h\|\nabla (u-u_h)\|_{L^\infty(0,T; L^2(\Omega))}+ h\left \|\nabla u_{t}-u_{h,t} \right\|_{L^\infty(0,T; L^2(\Omega))}  \leq \, C h^{s}. \end{multlined}
\end{equation*}
where $\max \{1, d/2\} < s \leq 2$. Our results are intended to enhance the numerical analysis of strongly damped quasilinear wave equations where the nonlinearities in the equation involve the time derivatives of the solution. We note that a particular feature of the present quasilinear equation is that the non-degeneracy is \emph{not} a priori given. In our proofs we have to ensure that the factor $1-2ku_h$ next to the second time derivative remains positive.\\
\indent In the continuous analysis of the Westervelt equation, non-degeneracy is typically achieved by a higher-regularity result for the solution and the use of an embedding, e.g., $H^2(\Omega) \hookrightarrow L^\infty(\Omega)$; see~\cite[Theorem 3.1]{KaltenbacherLasiecka_Westervelt}. Such a strategy is not possible here since we use piecewise linear basis functions. Instead we employ inverse estimates for finite element functions and the stability and approximation properties of the Scott--Zhang interpolation operator~\cite{scott1990finite}. \\
\indent Our analysis relies on the Banach fixed-point theorem combined with error estimates for a linear wave equation with variable coefficients. Therefore, in this work, we also obtain error estimates for strongly damped variable coefficient wave equations that take coefficient error into account as relevant, e.g., in optimal control problems in nonlinear acoustics~\cite{clason2009boundary, kaltenbacher2016shape, muhr2017isogeometric}. \\
\indent The rest of the paper is organized as follows. Section~\ref{Section:TheoreticalPreliminaries} introduces the notation and lays out the most important theoretical results in Sobolev and finite element spaces that we often use in the analysis. In Section~\ref{Section:Continuous problem}, we discuss the continuous problem and its well-posedness. In Section~\ref{Section:LinVariableCoeff}, we then study a linearized Westervelt equation with variable coefficients and prove that its semi-discretization has a unique solution. Section~\ref{Section:LinVariableCoeff_APriori} focuses on the a priori analysis of this linear model. In Section~\ref{Section:FEMWestervelt}, we show well-posedness and derive convergence rates for the semi-discrete Westervelt equation. Finally, Section~\ref{Section:NumExample} contains numerical examples that illustrate our theory.
\section{Theoretical preliminaries} ~\label{Section:TheoreticalPreliminaries} We begin by setting the notation and summarizing some auxiliary properties of Sobolev and finite element spaces that we will frequently use in the analysis.
\subsection{Notation} We denote the standard $L^2$ inner product by $(\cdot, \cdot)$. The norms in Sobolev spaces $L^p(\Omega)$ and $W^{q, p}(\Omega)$ are denoted by $|\cdot|_{L^p}$ and $|\cdot|_{W^{q, p}}$, respectively, where $1 \leq p \leq \infty$, $1 \leq q < \infty$. The norms in Bochner spaces $W^{q, p}(0,T; W^{r, s}(\Omega))$ are denoted by $\|\cdot\|_{W^{q, p}W^{r,s}}$, where $0 \leq q, r < \infty$, $1 \leq p, s \leq \infty$. We also introduce the spaces $\dot{H}^s(\Omega)=H_0^1(\Omega) \cap H^s(\Omega)$, for $1 \leq s \leq 2$.\\
\indent The constants $0< C_i < \infty$, $i \in \mathbb{N}$, appearing in the estimates denote generic constants that might depend on the coefficients in the equation and the domain $\Omega$, but not on the mesh size. Throughout the paper, we assume $T>0$ to be a fixed time horizon.
\subsection{Auxiliary inequalities} Let $\Omega \subset \mathbb{R}^d$, where $d \in \{1, 2 ,3\}$, be a bounded domain with Lipschitz regular boundary.  The nonlinear terms appearing in the Westervelt equation are of a quadratic type, so after variational testing, we often have to employ H\"older's inequality for a product of three functions. In particular, we frequently make use of the following three special cases of H\"older's inequality:
\begin{equation*}
    \begin{aligned}
        |fgh|_{L^1} \leq& \, |f|_{L^p}|g|_{L^q}|h|_{L^r} \quad \text{for  } f \in L^p(\Omega), \ g \in L^q(\Omega), \ h \in L^r(\Omega), 
    \end{aligned}
\end{equation*}
with $(p,q,r) \in \{(2, 4, 4), (3, 6, 2), (\infty, 2, 2)\}$. We also often employ a special case of Young's $\varepsilon$-inequality in the form
\begin{align} \label{Young_inequality}
    xy \leq \varepsilon x^2+\frac{1}{4 \varepsilon} y^2, \quad \text{where} \ x,y>0, \, \varepsilon>0;
\end{align}
see~\cite[Appendix B]{Evans}. Let $u$ and $v$ be non-negative continuous functions and $C_1,C_2<\infty$ non-negative constants such that
$$u(t) + v(t) \leq C_1+C_2 \int_0^t u(s) \, \textup{d}s \quad \text{ for all } t\in [0,T].$$ Then the following modification of Gronwall's inequality holds
\begin{equation} u(t) + v(t) \leq C_1 e^{C_2 T} \quad \text{ for all } t\in[0,T];
\label{Gronwall_inequality}
\end{equation}
see~\cite[Lemma 3.1]{garcke2017well}. Finally, we recall the Sobolev embeddings
\begin{equation} \label{embed_constants}
\begin{aligned}
 f \in H_0^1(\Omega) &\hookrightarrow L^p(\Omega), && \quad |f|_{L^p} \leq C_{H_0^1, L^p} |f|_{H^1},  \\
\end{aligned}
\end{equation}
for $1 \leq p \leq 6$, where $C_{H_0^1, L^p}< \infty$, noting that $d\leq 3$. 
\subsection{Finite element spaces} We consider the discretization in space by continuous piecewise linear finite elements that vanish on the boundary. Let $\Omega \subset \mathbb{R}^d$, $d \in \{2,3\}$, be a convex polygonal domain. For $h \in (0, \overline{h}]$, let $\mathcal{T}_h$  be a triangulation of $\Omega$ made of triangles (in $\mathbb{R}^2$) or of tetrahedrons (in $\mathbb{R}^3$) so that $\Omega=\cup_{K \in \mathcal{T}_h} K$. We denote by $P_{1}(K)$ the space of polynomials on $K$ of degree no greater than $1$. We introduce the finite element space as
\begin{align} \label{FEM_space}
    S_h=\{u_h \in H_0^1(\Omega): \ u_{h} \vert_{K}\, \in P_1(K), \, \forall K \in \mathcal{T}_h \}.
\end{align}
We assume that $\{\mathcal{T}_h\}_{0<h \leq \overline{h}}$ is a quasiuniform family: there are constants $0< c_1, c_2 < \infty$ such that
$$c_1 h \leq h_K \leq c_2 \varrho_K, \quad K \in \mathcal{T}_h,$$
where $h_K$ denotes the diameter of the triangle (tetrahedron) K, $\varrho_K$ stands for the diameter of the greatest ball (sphere) included in K, and $h= \max_{K \in \mathcal{T}_h} h_K.$ \\
\indent It is known that there exists $\mathcal{U} \in S_h$ and $0< C <\infty$ such that
\begin{equation} \label{approx_property_1}
\begin{aligned} 
 |u-\mathcal{U}|_{L^2} \leq& \, Ch^{s}|u|_{H^s}, \\
 |\nabla (u-\mathcal{U})|_{L^2} \leq& \, C h^{s-1}|u|_{H^s},
\end{aligned}
\end{equation}
for $u \in \dot{H}^{s}(\Omega)$, $1 \leq s \leq 2$; see~\cite{girault2012finite}. \\[2mm]
\noindent \textbf{Inverse estimates.}  Under the assumptions made above on the family $\{S_h\}_{0<h \leq \overline{h}}$, there is a $0< C_{\textup{inv}}< \infty$ such that 
\begin{align} \label{inverse_estimate}
& |\chi|_{L^\infty} \leq C_{\textup{inv}} h^{-d/p}|\chi|_{L^p}, \quad 1 \leq p < \infty,
\end{align}
for every $\chi \in S_h$; see~\cite[Theorem 4.5.11]{brenner2007mathematical}. We will need the special cases $p=2$ and $p=4$ in the proofs. \\[2mm]
\noindent \textbf{Bounds for the interpolation error.} In our analysis, we will employ an interpolant $I_h: W^{l, p}(\Omega) \rightarrow S_h$  of Scott--Zhang type, where $0 \leq l \leq 1$, $1 \leq p \leq \infty$; cf.~\cite{brenner2007mathematical, scott1990finite}. The following approximation and stability properties hold:
\begin{equation} \label{est_interpolant}
\begin{alignedat}{2}
   |v-I_h v|_{L^2} \leq& \, C_{\textup{app}} h^s |v|_{H^{s}}, \quad && \text{for } v \in H^s(\Omega), \ 1 \leq s \leq 2,\\
   |I_h v|_{L^\infty} \leq& \, C_{\textup{st}} |v|_{L^\infty}, \quad \ &&\text{for }  v \in L^\infty(\Omega),
\end{alignedat}
\end{equation}
where $0 < C_{\textup{app}}, C_{\textup{sta}} <\infty$; see~\cite[Theorem 4.8.12 and Corollary 4.8.15]{brenner2007mathematical}.
\section{The continuous problem} \label{Section:Continuous problem}
We start from the following initial-boundary value problem for the Westervelt equation
\begin{align} \label{Westervelt_continuous}
\begin{cases}
\displaystyle u_{tt}-c^2\Delta u-b \Delta u_{t}=2 k \left(u u_{tt}+ u_{t}^2 \right) \quad \text{in } \Omega \times (0,T),
\smallskip\\ u=0 \quad \text{ on }  \partial \Omega \times (0,T),  
\smallskip\\ \displaystyle \left(u, u_{t}\right)=(u_0, u_1) \quad \text{ on }   \Omega \times \{t=0\}.
\end{cases}
\end{align}
The weak form of the problem is then given by
\begin{align} \label{Westervelt_continuous_weak}
    \begin{cases}
\left( (1-2ku) u_{tt} , \phi \right) + c^2(\nabla u, \nabla \phi)+b\left(\nabla u_{t}, \nabla \phi \right)- 2k\left(u^2_{t}, \phi \right)=0,\smallskip \\
\text{for all } \phi \in H_0^1(\Omega) \text{ a.e. in time}, \smallskip\\
  \left (u(0), u_{t}(0) \right)=(u_{0}, u_{1}).
    \end{cases}
\end{align}
This problem is known to be well-posed for small data.
\begin{theorem}\label{wellposedness}\textup{ \cite[Theorem 3.1]{KaltenbacherLasiecka_Westervelt}}
Let $T>0$, $b, k, c^2>0$, $0<m < \frac{1}{4k}$, and $M>0$ be arbitrary. Assume that 
$$|u_0|^2_{H^2}+|u_1|^2_{H^2} \leq \varrho_T,$$ with $\varrho_T$ sufficiently small. Then there exists a unique solution $u$ of \eqref{Westervelt_continuous_weak} such that
\begin{equation} \label{ball_continuous}
\begin{aligned}
u  \in \mathcal{B}= \left\{ u \in L^{\infty}(\Omega \times (0,T)) :\right.  &\left. \, \|u\|_{L^\infty(\Omega \times (0,T))} \leq m, \, \|u_{tt}\|_{L^2 H^1} \leq M, \right. \\[1mm]
&  \left. \, \| u_{t} \|_{C H^1} \leq M, \ (u(0), u_t(0))=(u_0, u_1) \right\},
\end{aligned}
\end{equation}
and such that
$\Delta u,\ u_{tt},\ \nabla u_t  \in L^\infty(0,T; L^2(\Omega)), \ \nabla u_{tt} \in L^2(0,T; L^2(\Omega)).$
\end{theorem}
We also refer to \cite{meyer2011optimal} where the results of Theorem~\ref{wellposedness} are generalized by employing the maximal $L^p$ regularity approach. Results on the existence of very smooth solutions for a reformulation of the problem in terms of the acoustic velocity potential $\psi$, where $u= \varrho \psi_t$, can be found in~\cite{kaltenbacher2014efficient, kawashima1992global}. \\
\indent Note that the well-posedness holds for sufficiently small data which by continuity implies smallness of $u$ in the appropriate norms. The condition $\|u\|_{L^\infty(\Omega \times (0,T))} \leq m < 1/(4k)$ in \eqref{ball_continuous} ensures that the equation does not degenerate. For the well-posedness of the semi-discrete problem, we will also need smallness of data and a bound on the approximate solution that guarantees non-degeneracy. It is also worth noting that the strong damping (i.e., $b>0$) is needed for the continuous problem to be well-posed and the same will hold for the semi-discrete equation.  \\
\indent Going forward, we assume that \eqref{Westervelt_continuous} has a unique solution. We will impose additional conditions on the regularity of $u$ when needed for the convergence results.
\section{Finite element approximation of the linearized Westervelt equation with variable coefficients}\label{Section:LinVariableCoeff}
We first provide numerical analysis of an initial-boundary value problem for a linear wave equation with variable coefficients which can be interpreted as a linearization of the Westervelt equation. We study the following initial boundary value problem for a non-degenerate equation:
\begin{align} \label{Westervelt_lin_continuous}
\begin{cases}
\displaystyle \alpha(x,t)u_{tt}-c^2\Delta u-b \Delta u_t+\beta(x,t)u_t=f(x,t) \quad \text{in } \Omega \times (0,T),
\smallskip\\ u=0 \quad \text{ on }  \partial \Omega \times (0,T),  
\smallskip\\ \left(u, \displaystyle u_t \right)=(u_0, u_1) \quad \text{ on }   \Omega \times \{t=0\},
\end{cases}
\end{align}
where $0 < \alpha_0 \leq \alpha(x,t) \leq \alpha_1$ a.e. in $\Omega \times (0,T)$. Analysis of the linearization \eqref{Westervelt_lin_continuous} allows to later define an iterative map on which we will apply the Banach fixed-point theorem. However, finite element approximation of the partial differential equation in \eqref{Westervelt_lin_continuous} is also of independent interest. For example, this model with $b=f=0$ appears in~\cite{cavalcanti2002existence} and is motivated by the
study of the transonic gas dynamics. The adjoint problems for the Westervelt equation which arise in the optimal control and shape optimization works~\cite{clason2009boundary, kaltenbacher2016shape, muhr2017isogeometric} have (after time reversal) the form of this PDE as well. \\
\indent We refer to~\cite[Proposition 7.2]{kaltenbacher2014efficient} for the sufficient conditions under which problem \eqref{Westervelt_lin_continuous} has a unique solution $(u, u_t)$ in $C([0,T]; H^{j}(\Omega) \times H^{j-1}(\Omega))$, where $j \in \{2, 4\}$. We therefore proceed with the assumption that problem \eqref{Westervelt_lin_continuous} has a unique solution. The conditions on the regularity of $u$ are specified when needed for the a priori estimates. It is implicitly assumed that the coefficients $\alpha$ and $\beta$, the initial data $(u_0, u_1)$, and the source term $f$ are sufficiently smooth for such a regularity to hold.  \\
\indent Results on the error estimates for special cases of \eqref{Westervelt_lin_continuous} with constant coefficients are available in the literature. Analysis of the Galerkin approximation of \eqref{Westervelt_lin_continuous} for the case $\alpha=1$, $b=0 $, and $\beta=0$ is performed in~\cite{baker1976error}. The case of a strongly damped wave equation (i.e. with a fixed positive constant $b$) and with $\alpha=1$, $\beta=f=0$ is analyzed in~\cite{larsson1991finite, sinha2003effect, thomee2004maximum}. \\
\indent Let $\{S_h\}_{0<h \leq \overline{h}}$ be a family of subspaces of $H_0^1(\Omega)$ defined in \eqref{FEM_space} with basis $\{w_i\}_{i=1}^{N_h}$. We consider Galerkin approximations in space
\begin{align*}
& u_h(x,t)=\displaystyle \sum_{i=1}^{N_h} \xi_i(t) w_i(x), 
\end{align*}
where $\xi_i: (0,T) \to \mathbb{R}$ are coefficient functions for $i\in[1, N_h]$. Let $\alpha_h$, $\beta_h$, and $f_h$ be approximations of functions $\alpha$, $\beta$, and $f$, respectively, in $S_h$. 
\begin{assumption} \label{assumption_2} We assume that the approximate coefficients and the source term satisfy the following conditions \smallskip
\begin{itemize}
\item $\alpha_h \in L^\infty(0,T; L^\infty(\Omega))$, \ $\exists \ \alpha_0: \, \alpha_h \geq \alpha_0>0$ a.e. in $\Omega \times (0,T)$, \smallskip
\item $\beta_h \in L^\infty(0,T; L^3(\Omega))$,\smallskip
\item $f_h \in L^2(0,T; L^2(\Omega))$. 
\end{itemize}
\end{assumption}
\noindent For a given $h \in (0, \overline{h}]$, we next study a semi-discretization of \eqref{Westervelt_lin_continuous} in $S_h$ and prove that it has a unique solution.
\begin{theorem}\label{thm:existence_linear} Let $c^2$, $b>0$ and let Assumption~\ref{assumption_2} hold. For each $h \in (0, \overline{h}]$, there exists a unique function $u_h \in H^2(0,T; S_h)$ which satisfies 
\begin{equation} \label{weak_form_discrete}
\begin{aligned}
\left( \alpha_h u_{h, tt} , \phi \right) + c^2(\nabla u_h, \nabla \phi)+b\left(\nabla u_{h, t}, \nabla \phi \right)+ \left(\beta_h u_{h, t}, \phi \right)  =(f_h, \phi),
\end{aligned}
\end{equation}
for all $\phi \in S_h$, a.e. in time, and
\begin{equation} \label{discrete_initial_data}
\begin{aligned}
&   \left (u_h(0), u_{h, t}(0) \right)=(u_{h,0}, u_{h,1}),
\end{aligned}
\end{equation}
where $u_{h,0}$ and $u_{h,1}$ are approximations of $u_0$ and $u_1$ in $S_h$. Moreover, the following a priori bound holds
\begin{equation} \label{energy_est_lin}
\begin{aligned}
        & \left\|u_{h, tt}\right\|^2_{L^2L^2}+ \|\nabla u_h\|^2_{L^\infty L^2}+\left\|\nabla u_{h, t}\right\|^2_{ L^\infty L^2} +\left\|\nabla u_{h, t}\right\|^2_{L^2L^2}  \\[1mm] 
          \leq& \, C(\alpha_h, \beta_h, T) \,\left(|\nabla u_{h,0}|^2_{L^2}+\left|\nabla u_{h,1} (0)\right|^2_{L^2}+\|f_h\|^2_{L^2L^2}\right).
\end{aligned}
\end{equation}
The constant above is given by
\begin{equation} \label{Thm1_const}
\begin{aligned}
C(\alpha_h, \beta_h, T)= C_1\, \textup{exp}(C_2 \,(\|\alpha_h\|^2_{L^\infty L^3}+\| \beta_h\|_{L^\infty L^3}^2+1)\, T).
\end{aligned}
\end{equation}
\end{theorem}
\begin{proof}
The proof follows a general framework of the well-posedness proofs for the linearizations of the classical nonlinear acoustic equations that are based on the Galerkin approximations in space. In particular, we refer to~\cite[Theorem 1]{kaltenbacher2016shape} and \cite[Proposition 1]{fritz2018well}. However, for the continuous problem, the basis functions have to be in $H^2(\Omega)$ to later guarantee the non-degeneracy of the nonlinear model via the embedding $H^2(\Omega) \hookrightarrow L^\infty(\Omega)$. Our basis functions are only $H^1$ regular which changes the a priori estimates that we will derive. \\[2mm]
\noindent \textbf{Step 1: Existence of a solution.} We denote by $\xi_{h, 0}=[\xi_{1,0} \ldots \xi_{N_h, 0}]^T$ and $\xi_{h, 1}=[\xi_{1,1} \ldots \xi_{N_h, 1}]^T$ the components of the given initial approximations $u_{h,0}$ and $u_{h,1}$, respectively. Then our semi-discrete problem is to find $\xi_h=[\xi_1 \ldots \xi_{N_h}]^T$  such that
\begin{align} \label{matrix_equation}
\begin{cases}
\displaystyle M_h(t) \xi_{h, tt}+K_h \xi_h+C_h(t) \xi_{h, t}=F_h, \\
\xi_h(0)=\xi_{h, 0},  \\
\displaystyle \xi_{h, t}(0)=\xi_{h, 1},
\end{cases}
\end{align}
where the matrices are given by
\begin{equation*}
\begin{aligned}
M_h(t)=& \, [M_{ij}], \quad  && M_{ij}=(\alpha_h(t)\, w_i, w_j), \\
K_h=& \,[K_{ij}], \quad  && K_{ij}=c^2 \, (\nabla w_i,  \nabla w_j), \\
C_h(t)=& \, [C_{ij}], \quad   && C_{ij}=b \, (\nabla w_i, \nabla w_j)+ (\beta_h(t) \, w_i, w_j),
\end{aligned}
\end{equation*}
and the source term is given by $F_h=[F_1 \ldots F_{N_h}]^{T}$,  $F_j=(f_h, w_j)$, with $1 \leq i,j \leq N_h$. Note that the matrices and the right-hand side vector are all well-defined since
\begin{align*}
&|(\alpha_h\, w_i, w_j)| \leq |\alpha_h|_{L^2}|w_i|_{L^4}|w_j|_{L^4}, \\
&|(\beta_h \, w_i, w_j)| \leq |\beta_h|_{L^2}|w_i|_{L^4}|w_j|_{L^4}, \\
& |(f_h, w_i)|\leq |f_h|_{L^2}|w_i|_{L^2},
\end{align*}
a.e. in time. Furthermore, the matrix $M_h$(t) is invertible for a.e. $t \in [0,T]$; cf.~\cite[Theorem 1]{kaltenbacher2016shape}. The statement follows from the fact that $M_h(t)$ is positive definite. Indeed, for any $z \in \mathbb{R}^{N_h} \setminus \{0\}$, we have
\begin{align*} 
z^T M_h(t)z=  \bigintsss_{\Omega} \alpha_h(t) \left|\sum_{i=1}^{N_h} z_i w_i \right|^2 \, \textup{d}x \geq \alpha_0 \left|\sum_{i=1}^{N_h} z_i w_i \right|^2_{L^2}>0,
\end{align*}
for a.e. $t \in [0,T]$. Thanks to the fact that $M_h$ is invertible, the matrix equation in \eqref{matrix_equation} can be rewritten as
\begin{align*}
 \xi_{h, tt}+M_h^{-1}(t)C_h(t) \xi_{h, t}+M_h^{-1}(t) K_h \xi_h=M_h^{-1}(t) F_h.
\end{align*}
Now the existence of a solution $u_h \in H^2(0,T_h; S_h)$ follows from the standard ODE theory; see, for example,~\cite[Chapter 1]{roubivcek2013nonlinear}. To extend the existence interval to $[0,T]$, we next show that $u_h$ remains bounded on $[0,T]$ in the appropriate norms. \\[2mm]
\noindent \textbf{Step 2: A priori estimate.}  We want to derive an priori bound for $u_h$. To this end, we test our problem with two different test functions.  We first test \eqref{weak_form_discrete} with $\phi=\lambda  u_{h, t} \in S_h$, where $\lambda>0$, and integrate with respect to time from $0$ to $t$, $t \leq T_h$. After some standard manipulations, this action results in 
\begin{equation} \label{est_1}
\begin{aligned}
&\lambda \frac{c^2}{2}|\nabla u_h(t)|^2_{L^2} +\lambda b\left\|\nabla u_{h, t} \right\|^2_{L^2L^2}\\
\leq & \, \lambda \frac{c^2}{2} |\nabla u_{h, 0}|^2_{L^2}+\eps \left\|u_{h, tt} \right\|^2_{L^2L^2}+\frac{1}{4 \varepsilon}\lambda^2C_{H_0^1, L^2}^2\|f_h\|^2_{L^2L^2}\\
&+\left(\frac{1}{4 \eps}C^2_{H_0^1, L^6}\lambda^2\left(\|\alpha_h\|^2_{L^\infty L^3}+C^2_{H_0^1, L^2}\|\beta_h\|^2_{L^\infty L^3}\right)+2\eps\right)\left\|\nabla u_{h, t}\right\|^2_{L^2L^2},
\end{aligned}
\end{equation}
where $\varepsilon>0$ and $\lambda>0$ will be conveniently chosen. To be able to bound the term $\displaystyle \left\|u_{h, tt} \right\|^2_{L^2L^2}$ that appears on the right-hand side above, we next test \eqref{weak_form_discrete} with $\phi= u_{h, tt} \in S_h$. After integrating over $(0, t)$, this action yields the second inequality
\begin{equation} \label{est_2}
\begin{aligned}
&(\alpha_0-2\varepsilon)\left \|u_{h, tt} \right\|^2_{L^2L^2}+\frac{b}{4}\left|\nabla u_{h, t}(t)\right|^2_{L^2}\\
\leq& \,\frac{c^2}{2}|\nabla u_{h,0}|^2_{L^2}+\frac{c^2+b}{2}|\nabla u_{h,1}|^2_{L^2}+\frac{c^4}{b} |\nabla u_h(t)|^2_{L^2}\\
&+\left(\frac{1}{4 \eps}C^2_{H_0^1, L^6}\|\beta_h\|^2_{L^\infty L^3}+c^2\right)\left\|\nabla u_{h, t}\right\|^2_{L^2L^2}+\frac{1}{4 \eps}\|f_h\|^2_{L^2L^2}.
\end{aligned}
\end{equation}
\indent Above, we have estimated the $c^2$ term by first integrating by parts with respect to time and then employing H\"older's inequality and Young's $\varepsilon$-inequality with $\varepsilon \in \{ b/4, 1/2\}$:
\begin{equation} \label{tricky_term}
\begin{aligned}
&-c^2\int_0^t \int_{\Omega} \nabla u_h \cdot \nabla u_{h, tt} \, \textup{d}x \textup{d}s \\
=& \,-c^2 \int_{\Omega} \nabla u_h(s) \cdot \nabla u_{h, t}(s) \, \textup{d}x \, \Bigr \vert_0^t+c^2\int_0^t \int_{\Omega} \left|\nabla u_{h, t}\right|^2 \, \textup{d}x \textup{d}s \\
\leq&\, c^2 |\nabla u_h(t)|_{L^2}\left|\nabla u_{h, t}(t) \right|_{ L^2}+c^2 |\nabla u_{h, 0}|_{L^2}\left|\nabla u_{h, 1}\right|_{L^2}+c^2\left\|\nabla u_{h, t}\right\|^2_{L^2L^2}\\
\leq&\, \frac{c^4}{b} |\nabla u_h(t)|^2_{L^2}+\frac{b}{4}\left|\nabla u_{h, t}(t) \right|^2_{L^2}+\frac{c^2}{2} |\nabla u_{h, 0}|^2_{L^2}+\frac{c^2}{2}\left|\nabla u_{h, 1}\right|^2_{L^2}\\
&+c^2\left\|\nabla u_{h, t}\right\|^2_{L^2L^2}.
\end{aligned}
\end{equation}
To absorb $\displaystyle \frac{c^4}{b} |\nabla u_h(t)|^2_{L^2}$ by the corresponding term on the left side in \eqref{est_1}, we need to choose $\lambda>0$ sufficiently large so that $\lambda c^2/2>c^4/b$. By adding the derived inequalities \eqref{est_1} and \eqref{est_2}, we then obtain
\begin{equation} \label{est_before_Gronwall}
\begin{aligned}
 &  (\alpha_0-3\varepsilon) \left\|u_{h, tt}\right\|^2_{L^2L^2}+ \left(\lambda \frac{c^2}{2}-\frac{c^4}{b}\right)|\nabla u_h(t)|^2_{L^2}+\frac{b}{4}\left|\nabla u_{h, t} (t)\right|^2_{L^2}\\
 &+\lambda b\left\|\nabla u_{h, t}\right\|^2_{L^2L^2}\\
 \leq& \, (\lambda+1) \frac{c^2}{2}|\nabla u_{h, 0}|^2_{L^2}+\frac{c^2+b}{2}\left|\nabla u_{h, 1} \right|^2_{L^2}+\frac{1}{4 \eps}\left(\lambda^2 C^2_{H_0^1, L^2}+1\right)\|f_h\|^2_{L^2L^2}\\
 &\!\begin{multlined}[t]+\left\|\nabla u_{h, t}\right\|^2_{L^2L^2}\left(2\eps+c^2+\frac{1}{4 \eps}C^2_{H_0^1, L^6}\left(1+\lambda^2C^2_{H_0^1, L^2}\right)\|\beta_h\|^2_{L^\infty L^3}\right.\\
 \left.+ \frac{1}{4 \eps}C^2_{H_0^1, L^6}\lambda^2\|\alpha_h\|^2_{L^\infty L^3} \right).\end{multlined}
\end{aligned}
\end{equation}
We choose $\lambda=4c^2/b$ and $\varepsilon=\alpha_0/6$, apply Gronwall's inequality to \eqref{est_before_Gronwall}, and take the essential supremum over $t \in (0,T_h)$. In this way, we obtain
\begin{equation} \label{energy_estimate}
\begin{aligned}
    & \left\|u_{h, tt}\right\|^2_{L^2(0,T_h;L^2)}+ \|\nabla u_h\|^2_{L^\infty(0,T_h; L^2)}+\left\|\nabla u_{h, t}\right\|^2_{ L^\infty(0,T_h; L^2)} \\
    &+\left\|\nabla u_{h, t}\right\|^2_{L^2(0,T_h;L^2)} \\
          \leq& \, \!\begin{multlined}[t] C_1\, \textup{exp}\, \left (C_2 \, \left(\|\alpha_h\|^2_{L^\infty(0,T; L^3)}+\| \beta_h\|_{L^\infty (0,T;L^3)}^2+1 \right) T \right)\\
          \times \,\left(|\nabla u_{h,0}|^2_{L^2}+\left|\nabla u_{h,1} \right|^2_{L^2}
          +\|f_h\|^2_{L^2(0,T;L^2)} \right). \end{multlined}
\end{aligned}
\end{equation}
The right-hand side of \eqref{energy_estimate} does not depend on $T_h$, so we can show by an argument of contradiction that we are allowed to extend the existence interval of $u_h$ to $[0,T]$; i.e., $T_h=T$ and estimate \eqref{energy_est_lin} holds.
\end{proof}
\section{A priori estimates for the linearized Westervelt equation with variable coefficients} \label{Section:LinVariableCoeff_APriori} We now focus on proving a priori estimates for the linearized Westervelt equation that also take into account approximation error of the coefficients and the source term. We wish to estimate $u-u_h$. We follow the usual approach in the finite element analysis and split this difference into
\begin{align*}
u-u_h=\underbrace{u-R_h u}_{\varrho} +\underbrace{R_h u-u_h}_{\theta},
\end{align*}
where $R_h$ denotes the elliptic projection; cf.~\cite{brenner2007mathematical, thomee1984galerkin}. The idea is to rely on the existing results on elliptic projectors to bound $\varrho=u-R_h u$, whereas $\theta=R_h u - u_h$ will be seen as a solution of a wave PDE with a source term. By deriving estimates for this PDE, we will find a bound for $\theta$.
\subsection{Auxiliary results for the elliptic projection}  We first recall two useful results for an auxiliary elliptic problem for $\varrho$. We employ the Ritz projection $R_h$, i.e., the orthogonal projection with respect to the product $(\nabla u, \nabla \phi)$.
\begin{lemma}\textup{\cite[Lemma 2.1]{baker1976error}} \label{Lemma_projector_I}
Let $u$ be the solution of \eqref{Westervelt_lin_continuous}. Then there exists a unique mapping $R_h u \in L^2(0,T; S_h)$ which satisfies
\begin{align} \label{eq_projection}
(\nabla R_h u, \nabla \phi)=(\nabla u, \nabla \phi) \quad \text{for all } \phi \in S_h, \ t \geq 0.
\end{align}
Let $1 \leq p \leq \infty$. If for some integer $k \geq 0$, $\frac{\partial^k u}{\partial t^k} \in L^p(0,T; H^s(\Omega)),$ then $\frac{\partial^k R_h u}{\partial t^k} \in L^p(0,T; S_h)$,
and
\begin{align*} 
\left\|\frac{\partial^k }{\partial t^k} (u-R_h u)\right\|_{L^pL^2} \leq C h^s \left\|\frac{\partial^k u}{\partial t^k} \right \|_{L^p H^s},
\end{align*}
for some constant $C>0$ independent of $h$ and $u$, and $1 \leq s \leq 2$. 
\end{lemma}
\noindent We also need a bound on the gradient of $\varrho$ to be able to later derive $H^1$ bounds for $u-u_h$.
\begin{lemma}\label{Lemma_projector_II}  Let $u$ be the solution of \eqref{Westervelt_lin_continuous} such that $u \in L^p(0,T; H^s(\Omega))$, where $1 \leq s \leq 2$, $1 \leq p \leq  \infty$. Then it holds
\begin{equation}
\begin{aligned}
\|\nabla(u-R_h u)\|_{L^p L^2}\leq&\, Ch^{s-1}\|u\|_{L^p H^s}, \\
\end{aligned}
\end{equation}
\end{lemma}
\begin{proof}
The estimate follows directly from the Galerkin orthogonality of $u-R_h u$ a.e. in time, C\'ea's lemma and the approximation property \eqref{approx_property_1}. 
\end{proof}
We note that analogous bounds can be obtained for $\varrho_{t}$ and $\varrho_{tt}$ by differentiating \eqref{eq_projection} once and twice with respect to time.
\subsection{Bounds for \mathversion{bold}$\theta=R_h u-u_h$} Since we are able to estimate $\varrho$, we now focus on deriving two a priori bounds for $\theta$.  \\
\indent At this point, we choose the approximate initial data as Ritz projections of $u_0$, $u_1$ in order to have $\theta(0)=\theta_t(0)=0$. 
\begin{proposition}\label{Prop_lin_discrete}   Let $c^2$, $b>0$. Let $u$ be the solution of \eqref{Westervelt_lin_continuous} which satisfies
\begin{align*}
u \in L^\infty(0,T; \dot{H}^s(\Omega)), \quad u_t, \ u_{tt} \in L^2(0,T; \dot{H}^s(\Omega)),
\end{align*}
where $1 \leq s \leq 2$. Let Assumption~\ref{assumption_2} hold  and $\alpha_{h,t} \in L^\infty(0,T; L^3(\Omega))$. Furthermore, let $(u_{h,0}, u_{h,1})=(R_h u_0, R_h u_1)$. Then there exists a positive constant $C=C(\alpha_h, \beta_h, T)$ such that
\begin{equation} \label{lower_bound_theta}
\begin{aligned}
& \left\|\theta_{t}\right\|_{L^
\infty L^2}+ \|\nabla \theta\|_{L^
\infty L^2} + \left \|\nabla \theta_{t} \right\|_{L^2 L^2} \\[1mm]
\leq& \,\!\begin{multlined}[t] C \left \{ h^{s} \, \|u_t\|_{L^2 H^s}+h^{s} \, \|u_{tt}\|_{L^2 H^s}
+\|f-f_h\|_{L^2L^2} \right.\\[1.5mm]
\left.+\|\alpha-\alpha_h\|_{L^\infty L^2}\left\|u_{tt}\right\|_{L^2 L^3}+\|\beta-\beta_h\|_{L^\infty L^2}\left\|u_{t}\right\|_{L^2 L^3} \right\}. \end{multlined}
\end{aligned}
\end{equation}
\end{proposition}
\begin{proof}
\indent The main idea of the proof is to see $\theta$ as a solution of a wave equation with variable coefficients and a source term and then test that equation with a suitable test function. Compared to the similar results for solutions of linear wave equations with constant coefficients~\cite{baker1976error, larsson1991finite}, we also take into account the error of the varying coefficients. \\
\indent By subtracting the weak forms for $u$ and $u_h$ and recalling the definition of $R_h u$, we find that $\theta$ solves
\begin{equation} \label{equation_theta}
\begin{aligned}
& \left( \alpha_h \theta_{tt} , \phi \right) + c^2(\nabla \theta, \nabla \phi)+b\left(\nabla \theta_{t}, \nabla \phi\right)+\left(\beta_h \theta_{t}, \phi\right)\\
=& \, -\left(\alpha_h\varrho_{tt}, \phi \right)-\left(\beta_h \varrho_{t}, \phi \right)+(f-f_h, \phi)\\
&-\left((\alpha-\alpha_h)u_{tt}, \phi \right)-\left((\beta-\beta_h) u_{t}, \phi \right),
\end{aligned}
\end{equation}
for all $\phi \in S_h$ a.e. in time. We next want to test \eqref{equation_theta} with $\displaystyle \phi=\theta_{t}$, noting that $\theta_t \in S_h$ a.e. in time. To get optimal error estimates for $u_h$, it is important to only employ the $L^2$ spatial norm of $\varrho_{t}$ and $\varrho_{tt}$ in our estimates. We also have to pay special attention to estimating the last two terms on the right-hand side. Having in mind the nonlinear problem where we will know that
\[ \|\alpha -\alpha _h\|_{L^\infty L^2} +\|\beta-\beta_h\|_{L^\infty L^2}+h\|\nabla (\alpha -\alpha _h)\|_{L^\infty L^2}+h\|\nabla(\beta-\beta_h) \|_{L^\infty L^2}\leq C h^s,\]
we should only have the terms $\alpha-\alpha_h$ and $\beta-\beta_h$ in the $L^2$ spatial norm in our estimates to ensure optimal error rates for the nonlinear Westervelt equation. Testing \eqref{equation_theta} with $\theta_{t}$, integrating over $(0, t)$, and employing H\"older's inequality then results in
\begin{equation*}
\begin{aligned}
    & \frac{\alpha_0}{2} \left|\theta_{t}(t)\right|^2_{L^2}+\frac{c^2}{2} |\nabla \theta(t)|^2_{L^2} +b \left \|\nabla \theta_{t} \right\|^2_{L^2 L^2}  \\
\leq& \, \|\beta_h\|_{L^\infty L^3}\left(\left\|\theta_{t}\right\|_{L^2 L^2}+\left\|\varrho_{t}\right\|_{L^2 L^2} \right)\left\|\theta_{t}\right\|_{L^2 L^6}+\|\alpha_h\|_{L^\infty L^3} \left\|\varrho_{tt}\right\|_{L^2 L^2}\left\|\theta_{t}\right\|_{L^2 L^6} \\
&+\frac{1}{2}\left\|\alpha_{h,t} \right\|_{L^\infty L^3}\left \| \theta_{t} \right\|_{L^2 L^2}\left \| \theta_{t} \right\|_{L^2 L^6} +\|f-f_h\|_{L^2L^2}\left\|\theta_{t}\right\|_{L^2 L^2} \\
&+\left(\|\alpha-\alpha_h\|_{L^\infty L^2}\left\|u_{tt}\right\|_{L^2 L^3} \right.
\left.+\|\beta-\beta_h\|_{L^\infty L^2}\left\|u_{t}\right\|_{L^2 L^3}\right)\left\|\theta_{t}\right\|_{L^2 L^6},
\end{aligned}
\end{equation*}
for a.e. $t \in [0,T]$. Above, we have used the identity
\begin{align*}
 \displaystyle   \int_0^t \int_\Omega \alpha_h \theta_{tt} \theta_{t} \, \textup{d}x \, \textup{d}s=\frac12 \left (\int_{\Omega}\alpha_h(s)\left.\left |\theta_{t}(s)\right|^2 \, \textup{d}x \right ) \, \right \rvert_0^t-\frac12 \int_0^t \int_\Omega  \alpha_{h,t} \left|\theta_{t}\right|^2 \, \textup{d}x \, \textup{d}s,
\end{align*}
 and the fact that $\theta(0)=\theta_t(0)=0$. By further employing the embedding results \eqref{embed_constants} and Young's $\varepsilon$-inequality \eqref{Young_inequality} with $\varepsilon \in \{b/8, 1 \}$ to handle the product terms, we get
\begin{equation} \label{lower_est_theta}
\begin{aligned}
& \frac{\alpha_0}{2} \left|\theta_{t}(t)\right|^2_{L^2}+\frac{c^2}{2}  |\nabla \theta(t)|^2_{L^2} +\frac{b}{2} \left \|\nabla \theta_{t} \right\|^2_{L^2 L^2} \\
\leq& \, \left\|\theta_{t}\right\|^2_{L^2 L^2}+\frac{4}{b} C_{H_0^1, L^6}^2 \|\beta_h\|^2_{L^\infty L^3}\left(\left\|\theta_{t}\right\|^2_{L^2 L^2}+\left\|\varrho_{t}\right\|^2_{L^2 L^2} \right) +\frac{1}{4 } \|f-f_h\|^2_{L^2L^2} \\
&+\frac{1}{2b}C_{H_0^1, L^6}^2\left\|\alpha_{h,t} \right\|^2_{L^\infty L^3}\left \| \theta_{t} \right\|^2_{L^2 L^2}+\frac{2}{b}C_{H_0^1, L^6}^2\|\alpha_h\|^2_{L^\infty L^3} \left\|\varrho_{tt}\right\|^2_{L^2 L^2}\\
&+\frac{4}{b}C_{H_0^1, L^6}^2\left(\|\alpha-\alpha_h\|^2_{L^\infty L^2}\left\|u_{tt}\right \|^2_{L^2 L^3}+\|\beta-\beta_h\|^2_{L^\infty L^2}\left\|u_{t}\right\|^2_{L^2 L^3}\right).
\end{aligned}
\end{equation}
Applying Gronwall's inequality \eqref{Gronwall_inequality} to the above estimate and taking the essential supremum over $(0,T)$ then leads to 
\begin{equation*} 
\begin{aligned}
& \left\|\theta_{t}\right\|^2_{L^
\infty L^2}+ \|\nabla \theta\|^2_{L^
\infty L^2} + \left \|\nabla \theta_{t} \right\|^2_{L^2 L^2} \\
\leq& \,\overline{C}(\alpha_h, \beta_h, T)\!\begin{multlined}[t] \left\{ \|\beta_h\|^2_{L^\infty L^3}\left\|\varrho_{t}\right\|^2_{L^2 L^2} 
+ \|\alpha_h\|^2_{L^\infty L^3} \left\|\varrho_{tt}\right\|^2_{L^2 L^2}+\|f-f_h\|^2_{L^2L^2} \right.\\
\left.+\|\alpha-\alpha_h\|^2_{L^\infty L^2}\left \|u_{tt} \right\|^2_{L^2 L^3}+\|\beta-\beta_h\|^2_{L^\infty L^2}\left\|u_{t}\right\|^2_{L^2 L^3}\right \}, \end{multlined}
\end{aligned}
\end{equation*}
with $\overline{C}(\alpha_h, \beta_h, T)=C_3 \,\textup{exp}\left(C_4 \left(\left\|\alpha_{h,t} \right\|^2_{L^\infty L^3}+ \|\beta_h\|^2_{L^\infty L^3}+1\right)T \right)\,$. Thanks to the results on the elliptic projector stated in Lemma~\ref{Lemma_projector_I} which provide the upper bounds on $\|\varrho_t\|_{L^2L^2}$ and  $\|\varrho_{tt}\|_{L^2L^2}$, we then obtain the final bound \eqref{lower_bound_theta}, where the constant is given by
\begin{equation*}
\begin{aligned}
&C(\alpha_h, \beta_h, T) \\
=& \, C_5 \left(\|\alpha_{h}\|_{L^\infty L^3}+\|\beta_h\|_{L^\infty L^3} +1\right) \textup{exp}\left(C_6 \, \left(\left\|\alpha_{h,t} \right\|^2_{L^\infty L^3}+\|\beta_h\|^2_{L^\infty L^3}+1\right)T \right).
\end{aligned}
\end{equation*} 
\end{proof}
\begin{remark}
We note that if $\alpha_h \in C([0,T]; L^\infty (\Omega))$, the same error order can be obtained if we choose any $u_{1,h}$ such that $|u_1-u_{1,h}|_{L^2} \leq C h^s$. This is due to the fact that we would then have an additional term on the right-hand side of \eqref{lower_est_theta} that is of order $h^s$: $$|\alpha_h(0)|_{L^\infty}\left|\theta_{t}(0)\right|^2_{L^2} \leq C|\alpha_h(0)|_{L^\infty}(|u_1-u_{1,h}|^2_{L^2}+\|\varrho(0)\|^2_{L^2}).$$
\end{remark}
\indent To be able to later employ a fixed-point approach and derive an a priori estimate for the nonlinear model, we also need to bound $\|\theta_{tt}\|_{L^2L^2}$. 
\begin{proposition} \label{Prop:theta_higher} Let $c^2$, $b>0$ and let $u$ to be the solution of \eqref{Westervelt_lin_continuous} that satisfies 
\begin{equation} \label{condition_u_lin}
\begin{aligned}
&u  \in L^\infty(0,T; \dot{H}^s(\Omega)), \quad u_t \in L^2(0,T; L^\infty(\Omega)) \cap L^\infty(0,T; \dot{H}^s(\Omega)), \\
& u_{tt} \in L^2(0,T; L^\infty(\Omega) \cap \dot{H}^s(\Omega)),
\end{aligned}
\end{equation}
where $s \in [1,2]$. Let Assumption~\ref{assumption_2} hold and let
$\beta_h \in L^2(0,T; L^\infty(\Omega))$ and $\alpha_{h,t} \in L^\infty(0,T; L^3(\Omega))$. Assume that $(u_{h, 0}, u_{h, 1})=(R_h u_0, R_h u_1)$. Then there exists a positive constant $C=C(\alpha_h, \beta_h, T)$ such that
\begin{equation} \label{higher_est_theta}
\begin{aligned}
& \|\nabla \theta\|_{L^
\infty L^2}+\left\|\theta_t\right\|_{L^
\infty L^2} + \left \|\nabla \theta_t \right\|_{L^\infty L^2} +\left\|\theta_{tt}\right\|_{L^2 L^2} \\[0.5mm]
\leq& C\,\!\begin{multlined}[t] \{h^{s}  \left \|u_t \right \|_{L^\infty H^s}+h^{s} \|u_{tt}\|_{L^2 H^s}+\|f-f_h\|_{L^2L^2}\\[1mm]
+\left\|u_{tt}\right\|_{L^2 L^\infty}\|\alpha-\alpha_h\|_{L^\infty L^2}  +\left\|u_t\right\|_{L^2 L^\infty}\|\beta-\beta_h\|_{L^\infty L^2}\}.\end{multlined}
\end{aligned}
\end{equation}
\end{proposition}
\begin{proof}
\indent To obtain the higher order estimate, we additionally test \eqref{equation_theta} with $\displaystyle \phi=\theta_{tt}$. After integrating over $(0,t)$ and recalling that $\theta(0)=\theta_t(0)=0$, we find 
\begin{equation} \label{est_theta_tt}
\begin{aligned}
&(\alpha_0-6\eps) \left\|\theta_{tt}\right\|^2_{ L^2 L^2}+\frac{b}{4}\left|\nabla \theta_t (t)\right |^2_{L^2} \smallskip \\
\leq& \,\frac{1}{4\eps} \| \beta_h\|^2_{L^\infty L^3}C^2_{H_0^1, L^6} \left\|\nabla \theta_t\right\|^2_{L^2 L^2}+\frac{1}{4 \eps}\|\alpha_h\|^2_{L^\infty L^\infty}\left\|\varrho_{tt}\right\|^2_{L^2L^2},  \\
&+\frac{1}{4 \varepsilon}\|\beta_h\|^2_{L^2 L^\infty}\left \|\varrho_t\right\|^2_{L^\infty L^2}+\frac{1}{4 \eps}\|\alpha-\alpha_h\|^2_{L^\infty L^2}\left \|u_{tt}\right\|^2_{L^2 L^\infty} \\
&+\frac{1}{4 \eps}\|\beta-\beta_h\|^2_{L^\infty L^2} \left\|u_t\right\|^2_{L^2 L^\infty}+\frac{1}{4 \eps}\|f-f_h\|^2_{L^2L^2} + \frac{c^4}{b} |\nabla \theta(t)|^2_{L^2}\\
&+c^2\left\|\nabla \theta_t\right\|^2_{L^2L^2},
\end{aligned}
\end{equation}
for $t \in [0,T]$. Above we have estimated $\displaystyle c^2\int_0^{t}\int_{\Omega} \nabla \theta  \cdot \nabla \theta_{tt}\, \textup{d}x$ in the same manner as \eqref{tricky_term}. Since $\theta(0)=0$, we can further infer that
\begin{align*}
   |\nabla \theta(t)|_{L^2}= \left|\int_0^t \nabla \theta_t(s) \, \textup{d}s \right|_{L^2}\leq  \int_0^t |\nabla \theta_t(s)|_{L^2} \, \textup{d}s \leq \sqrt{T}\|\nabla\theta_t\|_{L^2L^2}, \ t \in[0,T].
\end{align*}
Note that, compared to Proposition~\ref{Prop_lin_discrete}, we had to introduce additional assumptions \eqref{condition_u_lin} on the $L^\infty$ regularity of $u_t$ and $u_{tt}$. This is again due to the fact that we do not want to have higher than $L^2$ spatial norms of $\alpha-\alpha_h$ and $\beta-\beta_h$ on the right-hand side of \eqref{est_theta_tt}.\\
\indent We choose $\varepsilon < \alpha_0/6$ and add \eqref{est_theta_tt} to  estimate \eqref{lower_est_theta} multiplied by $\lambda>0$ such that $\lambda c^2/2>c^4/b$. We then apply Gronwall's inequality to the resulting estimate to obtain 
\begin{equation*}
\begin{aligned}
& \left\|\theta_{tt}\right\|^2_{L^2 L^2}+\left\|\theta_t\right\|^2_{L^
\infty H^1}+ \left \|\nabla \theta_t \right\|^2_{L^2 L^2} +\|\nabla \theta\|^2_{L^\infty L^2}\\
\leq& \, \!\begin{multlined}[t] C_7 \, \textup{exp}\left(C_8\, \left(\left\|\alpha_{h, t} \right\|^2_{L^\infty L^3}+ \|\beta_h\|^2_{L^\infty L^3}+T+1\right)T \right)
 \,\left\{\|\beta_h\|^2_{L^2 L^\infty}\left\|\varrho_t\right\|^2_{L^\infty L^2} 
\right. \\
+ \|\beta_h\|^2_{L^\infty L^3}\left\|\varrho_t\right\|^2_{L^2 L^2}
+\|\alpha_h\|^2_{L^\infty L^\infty} \left\|\varrho_{tt}\right\|^2_{L^2 L^2} \vphantom{\left(\left\|u_{tt}\right\|^2_{L^\infty L^4}+\left\|u_{tt}\right\|^2_{L^\infty L^\infty}\right)}
+\left\|u_{tt}\right\|^2_{L^2 L^\infty}\|\alpha-\alpha_h\|^2_{L^\infty L^2}
\\
\left.+\left\|u_t\right\|^2_{L^2 L^\infty}\|\beta-\beta_h\|^2_{L^\infty L^2}+ \|f-f_h\|^2_{L^2L^2} \vphantom{\left(\left\|u_{tt}\right\|^2_{L^\infty L^4}+\left\|u_{tt}\right\|^2_{L^2 L^\infty}\right)} \right \}, \end{multlined}
\end{aligned}
\end{equation*}
where we have also used that $|v|_{L^3} \leq C_{L^\infty, L^3}|v|_{L^\infty}$, for $v \in L^\infty(\Omega)$. Employing the bounds on $\|\varrho_t\|_{L^\infty L^2}$ and $\|\varrho_{tt}\|_{L^2L^2}$ then leads to the estimate \eqref{higher_est_theta}, where the constant is given by
\begin{equation*}
\begin{aligned}
&C(\alpha_h, \beta_h, T) \\
=& \, C_9 \!\begin{multlined}[t] \left(\|\alpha_h\|_{L^\infty L^\infty}+\|\beta_h\|_{L^\infty L^3}+\|\beta_h\|_{L^2 L^\infty}+1\right)\\[-1.8mm]
\times \textup{exp}\left(C_{10} \left(\left\|\alpha_{h, t} \right\|^2_{L^\infty L^3}+\|\beta_h\|^2_{L^\infty L^3}+T+1\right)T \right). \end{multlined}
\end{aligned}
\end{equation*}
\end{proof}
\subsection{A priori estimate for the linear equation} We can now state the a priori estimate for the linearized Westervelt equation with variable coefficients.
\begin{theorem}\label{Thm:LinWestervelt} Let the assumptions of Proposition~\ref{Prop:theta_higher} hold. Then the following a priori estimate is satisfied
\begin{equation} \label{HigerAPriori_lin}
\begin{aligned}
& \|u-u_h\|_{L^\infty L^2}+\left\|u_t-u_{h,t} \right\|_{L^
\infty L^2}+\left\|u_{tt}-u_{h,tt} \right\|_{L^
2 L^2}\\[1mm]
&+h\|\nabla (u-u_h)\|_{L^\infty L^2}+ h\left \|\nabla (u_{t}-u_{h,t}) \right\|_{L^\infty L^2}  \\[1mm]
\leq& \,C(\alpha_h, \beta_h, T)\,\!\begin{multlined}[t] \left\{ h^{s} \|u\|_{L^\infty H^s}+h^{s}\|u_t\|_{L^\infty H^s}+h^{s}\|u_{tt}\|_{L^2 H^s}+\|f-f_h\|_{L^2L^2}\right. \\[1.5mm]
+\left\|u_{tt}\right\|_{L^2 L^\infty}\|\alpha-\alpha_h\|_{L^\infty L^2}  
+\left\|u_{t}\right\|_{L^2 L^\infty}\|\beta-\beta_h\|_{L^\infty L^2} \},\end{multlined}
\end{aligned}
\end{equation}
where $u_h$ solves \eqref{weak_form_discrete}, \eqref{discrete_initial_data}. The constant appearing above is given by
\begin{equation} \label{final_constant}
\begin{aligned}
& C(\alpha_h, \beta_h, T)\\
=&\, C_{11}\!\begin{multlined}[t] \left\{\vphantom{\left\|\alpha_{h,t} \right\|^2_{L^\infty L^3}} \left(\|\alpha_h\|_{L^\infty L^\infty}+\|\beta_h\|_{L^2 L^\infty}+ \|\beta_h\|_{L^\infty L^3} +1\right) \,
 \right.\\
\left.\,\times \textup{exp}\, \Bigl(C_{12} \left(\left\|\alpha_{h,t} \right\|^2_{L^\infty L^3}+\|\beta_h\|^2_{L^\infty L^3}+T+1\right) T \Bigr)+1 \vphantom{\left\|\alpha_{h,t} \right\|^2_{L^\infty L^3}} \right\}.\end{multlined}
\end{aligned}
\end{equation}
\end{theorem}
\begin{proof}
The estimate follows directly by splitting the difference $u-u_h$ into the $\theta$ and $\varrho$ terms, and then employing Proposition~\ref{Prop:theta_higher}, Lemma~\ref{Lemma_projector_I}, and Lemma~\ref{Lemma_projector_II}, and the fact that
$$\|\nabla \varrho_t\|_{L^\infty L^2} \leq C h^{s-1} \|u_t\|_{L^\infty H^s},$$ for some constant $C>0$ independent of $h$ and $u$.
\end{proof}
We note that Theorem~\ref{Thm:LinWestervelt} also includes, as a special case where $\alpha=1$, $\beta=f=0$, the a priori estimate for strongly damped linear wave equations with constant coefficients; this result corresponds to~\cite[Theorem 3.4]{larsson1991finite}. If we do not have to take the coefficient error into account, regularity conditions \eqref{condition_u_lin} for $u$ can be relaxed to 
\begin{equation*}
\begin{aligned} 
u, \ u_t \in L^\infty(0,T; \dot{H}^s(\Omega)), \quad  u_{tt} \in L^2(0,T; \dot{H}^s(\Omega)),
\end{aligned}
\end{equation*}
since we lose the last two terms in the estimate \eqref{HigerAPriori_lin}. 
\section{Finite element approximation of Westervelt's equation}\label{Section:FEMWestervelt}
We are now ready to study discretization in space of the initial-boundary value problem \eqref{Westervelt_continuous} for the Westervelt equation. We want to prove that it has a unique solution in a neighbourhood of $u$. To this end, we rely on the Banach fixed-point theorem. 
\begin{theorem}\label{thm:Westervelt} \textup{\textbf{[A priori error estimate]}} Let $c^2$, $b$, $k>0$, and $T>0$. Assume that the initial-boundary value problem \eqref{Westervelt_continuous_weak} for the Westervelt equation has a unique solution which satisfies
\begin{align*}
&u \in L^\infty(0,T; L^\infty(\Omega) \cap \dot{H}^s(\Omega)), \ u_t \in L^2(0,T; L^\infty(\Omega)) \cap L^\infty(0,T; \dot{H}^s(\Omega)),\\
& u_{tt} \in L^2(0,T; L^\infty(\Omega) \cap \dot{H}^s(\Omega)),
\end{align*}  where $\max \{1, d/2\} < s \leq 2$. Then for sufficiently small
\begin{align*}
 & m=\|u\|_{L^\infty L^\infty}, \\
 & M=\max \left\{\|u\|_{L^\infty H^s}, \left \|u_t \right \|_{L^\infty H^s}, \left\|u_t \right \|_{L^2 L^\infty}, \left\|u_{tt} \right \|_{L^2 H^s}, \left\|u_{tt} \right \|_{L^2 L^\infty}  \right\},
\end{align*}
and $h$, there exists a unique $u_h \in H^2(0,T; S_h)$ in a neighbourhood of $u$ which satisfies equation 
\begin{equation} \label{Westervelt_semi-discrete}
\begin{aligned}
\displaystyle \left( (1-2ku_h)u_{h, tt} , \phi \right) + c^2(\nabla u_h, \nabla \phi)+b\left(\nabla u_{h, t}, \nabla \phi \right)= 2k\left(u_{h, t}^2, \phi \right),
\end{aligned}
\end{equation}
for all $\phi \in S_h$ a.e. in time, and $\left(u_{h}(0), u_{h, t}(0)\right)=\left(R_h u_0, R_h u_1 \right).$ Furthermore, there exists a positive constant $C$ that depends on $m$, $M$, and $T$, but not on $h$, such that 
\begin{equation} \label{a_priori_bound}
\begin{multlined}[c] \|u-u_h\|_{L^\infty L^2}+\left\|u_t-u_{h,t} \right\|_{L^
\infty L^2}+\left\|u_{tt}-u_{h,tt} \right\|_{L^
2 L^2}\\[2mm]
+h\|\nabla (u-u_h)\|_{L^\infty L^2}+ h\left \|\nabla u_{t}-u_{h,t} \right\|_{L^\infty L^2}  \leq \, C h^{s}. \end{multlined}
\end{equation}
\end{theorem}
\begin{proof}
The main idea of the proof is to define an iterative map on which we will apply the Banach fixed-point theorem while relying on the results for the linearized problem. We first introduce the set
\begin{equation*}
\begin{aligned}
\mathcal{B}_h=\left \{ \vphantom{\left\|u_{tt}-v_{h, tt} \right\|_{L^
2 L^2}}v_h \in X_h\, :\right. & \begin{multlined}[t]\left.\left\|u_{tt}-v_{h, tt} \right\|_{L^
2 L^2}+\left\|u_t-v_{h, t} \right\|_{L^
\infty L^2}+\|u-v_h\|_{L^\infty L^2} \right.\\[1.5mm]
+h\|\nabla (u-v_h)\|_{L^\infty L^2}+ h\left \|\nabla (u_t-v_{h,t}) \right\|_{L^\infty L^2} \leq \, L h^{s}, \\[1mm]
\left.\left(v_{h}(0), v_{h, t}(0)\right)=\left(R_h u_0, R_h u_1 \right)\, \vphantom{\left\|u_{tt}-v_{h, tt} \right\|_{L^
2 L^2}}\right\}, \end{multlined}
\end{aligned}
\end{equation*}
the constant $L>0$ is independent of $h$ and $h \leq \overline{h}$. Note that the set $\mathcal{B}_h$ is non-empty since $R_h u \in \mathcal{B}_h$. For $v_h \in \mathcal{B}_h$, we then consider the following linearization of our semi-discrete problem
\begin{equation} \label{Westervelt_semi-discrete_lin}
\begin{aligned}
\begin{cases}
\hspace{1.6mm} \left( (1-2kv_h)u_{h, tt} , \phi \right) + c^2(\nabla u_h, \nabla \phi)+b\left(\nabla u_{h, t}, \nabla \phi \right)\\[1mm]
= 2k\left(v_{h,t} \, u_{h, t}, \phi \right),  \\[1mm]
\text{for every $\phi \in S_h$, pointwise a.e. in $(0,T)$},  \\[1mm]
\left(u_{h}(0), \tfrac{\partial u_h}{\partial t}(0)\right)=\left(R_h u_0, R_h u_1 \right).
\end{cases}
\end{aligned}
\end{equation} We introduce an iterative map $\mathcal{F}: v_h \mapsto u_h$, noting that $\mathcal{F}$ will be well-defined thanks to the uniqueness result from Theorem~\ref{thm:existence_linear}. Clearly a fixed point of $\mathcal{F}$ would solve  \eqref{Westervelt_semi-discrete}, so we proceed with verifying the conditions of the Banach fixed-point theorem. \\ [2mm]
\noindent \textbf{$(\mathcal{B}_h, d)$ is a complete metric space.}  $\mathcal{B}_h$ is closed in $\mathcal{X}=W^{1, \infty}(0,T; H_0^1(\Omega)) \cap H^2(0,T; L^2(\Omega))$ with respect to the metric $d$ induced by $$||| v |||_{\mathcal{X}}:=\max \left \{\|v_{tt}\|_{L^2L^2}, \|v\|_{W^{1, \infty} H_0^1} \right \},$$ 
from which completeness follows. \\[2mm]
 \noindent \textbf{$\mathcal{F}$ is a self-mapping.} Take any $v_h \in \mathcal{B}_h$. We want to show that $u_h=\mathcal{F} v_h \in \mathcal{B}_h$. We set $$\alpha_h(x,t)=1-2k v_h, \quad \displaystyle \beta_h=-2k v_{h,t},\quad f_h=0,$$ and check that the conditions of Theorem~\ref{thm:existence_linear} and Theorem~\ref{Thm:LinWestervelt} are satisfied. We first show that the non-degeneracy condition on $
 \alpha_h$ is fulfilled. Employing the identity $v_h=v_h-I_hu+I_hu$ and relying on the inverse estimate \eqref{inverse_estimate} yields
\begin{equation*}
    \begin{aligned}
        \|v_h\|_{L^\infty L^\infty} \leq& \,\|v_h-I_h u\|_{L^\infty L^\infty}+\|I_h u\|_{L^\infty L^\infty}\\
        \leq& \, h^{-d/2}C_{\textup{inv}}\|v_h-I_h u\|_{L^\infty L^2}+\|I_h u\|_{L^\infty L^\infty}.
    \end{aligned}
\end{equation*}
Then by additionally using the identity $v_h-I_hu=v_h-u+u-I_hu$ and the stability and approximation properties \eqref{est_interpolant} of the interpolant, we conclude that
\begin{equation} \label{Linf_bound1_b}
    \begin{aligned}
     \|v_h\|_{L^\infty L^\infty} 
        \leq& \, h^{-d/2} C_{\textup{inv}}\|v_h-u\|_{L^\infty L^2}+ h^{-d/2}C_{\textup{inv}}\|u-I_h u\|_{L^\infty L^2} \\
          &+C_{\textup{st}}\|u\|_{L^\infty L^\infty}\\[1mm]
        \leq&\, C_{\textup{inv}}L h^{s-d/2}+C_{\textup{inv}} C_{\textup{app}} h^{s-d/2} \|u\|_{L^\infty H^s} +C_{\textup{st}}\|u\|_{L^\infty L^\infty}\\[1mm]
         \leq&\,  C_{\textup{inv}}L \overline{h}^{s-d/2}+C_{\textup{inv}} C_{\textup{app}} \overline{h}^{s-d/2} M +C_{\textup{st}}m.
    \end{aligned}
\end{equation}
By choosing $\overline{h}$, $M$, and $m$ sufficiently small so that
$$m_0:= C_{\textup{inv}}L \overline{h}^{s-d/2}+C_{\textup{inv}} C_{\textup{app}} \overline{h}^{s-d/2} M +C_{\textup{st}}m < 1/(2k),$$
we can guarantee that $ \|v_h\|_{L^\infty L^\infty}\leq m_0 < 1/(2k)$. In this way, the non-degeneracy condition in Assumption~\ref{assumption_2} is fulfilled since
\begin{align*}
&    0< \alpha_0=1-2km_0 \leq \alpha_h (x,t) \leq \alpha_1=1+2k m_0 \quad \text{in } \ \Omega \times(0,T).
\end{align*}
We next want to bound $\|\beta_h\|_{L^2 L^\infty}$, $\|\beta_h\|_{L^\infty L^3}$, $\|\alpha_h\|_{L^\infty L^3}$, and $\|\alpha_{h,t}\|_{L^\infty L^3}$ uniformly with respect to $h$. Similarly to \eqref{Linf_bound1_b}, we derive the following estimate
\begin{equation} \label{Linf_bound2}
\begin{aligned}
     \left\|v_{h,t}\right\|_{L^2 L^\infty} 
     \leq &\, \left\|v_{h,t}-I_h u_t\right\|_{L^2 L^\infty} +\left\|I_{h} u_t\right\|_{L^2 L^\infty} \\[1mm]
       \leq& \, C_{\textup{inv}}h^{s-d/2}\left\|v_{h,t}-u_t+u_t-I_h u_t\right\|_{L^2 L^2}+\left\|I_{h} u_t\right\|_{L^2 L^\infty}  \\[1mm]
     \leq&\, \sqrt{T} C_{\textup{inv}}L h^{s-d/2}+C_{\textup{inv}} C_{\textup{app}} h^{s-d/2} \left\|u_t \right\|_{L^2 H^s}
     +C_{\textup{st}}\left\|u_t\right\|_{L^2 L^\infty},
\end{aligned}    
\end{equation}
from which we have $ \left\|\beta_{h}\right\|_{L^2 L^\infty} \leq 2k (\sqrt{T} C_{\textup{inv}}L \overline{h}^{s-d/2}+C_{\textup{inv}} C_{\textup{app}} \overline{h}^{s-d/2} M+M) $.
Furthermore, it holds that
\begin{align*}
 & \left\|\beta_{h}\right\|_{L^\infty L^3}=\left\|2k v_{h,t}\right\|_{L^\infty L^3} \leq 2k C_{H_0^1, L^3}(L\overline{h}^{s-1}+C M), \\[0.5mm]
 &  \left\|\alpha_{h}\right\|_{L^\infty L^3}=\left\|1-2kv_h \right\|_{L^\infty L^3} \leq C(\Omega)+ 2kC_{H_0^1, L^3}(L\overline{h}^{s-1}+C M),
\end{align*}
and $$\left\|\alpha_{h, t}\right\|_{L^\infty L^3}=\left\|\beta_{h}\right\|_{L^\infty L^3} \leq 2k C_{H_0^1, L^3}(L\overline{h}^{s-1}+C M).$$ According to Theorem~\ref{thm:existence_linear}, there exists a unique solution $u_h \in X_h$ of \eqref{Westervelt_semi-discrete_lin}. Thanks to the a priori bound for the linearized Westervelt equation stated in Theorem~\ref{Thm:LinWestervelt}, we have 
\begin{equation} \label{est_fixedpoint_}
\begin{aligned}
& \|u-u_h\|_{L^\infty L^2}+\left\|u_t-u_{h,t} \right\|_{L^
\infty L^2}+\left\|u_{tt}-u_{h,tt} \right\|_{L^
2 L^2}\\[1mm]
&+h\|\nabla (u-u_h)\|_{L^\infty L^2}+ h\left \|\nabla u_{t}-u_{h,t} \right\|_{L^\infty L^2}  \\[1mm]
\leq&\, C_{*} \begin{multlined}[t] \left\{h^{s} \|u\|_{L^\infty H^s}+h^{s} \|u_t\|_{L^\infty H^s}+h^{s} \|u_{tt}\|_{L^2 H^s} \right.  \\[1mm]
\left.+ k \left\|u_{tt}\right\|_{L^2 L^\infty}\|u-v_h\|_{L^\infty L^2}+ k\left\|u_t\right\|_{L^2 L^\infty}\left\|u_t-v_{h,t} \right\|_{L^\infty L^2} \right \}, \end{multlined}
\end{aligned}
\end{equation}
where the constant appearing above is computed according to \eqref{final_constant} and the derived uniform bounds on $\alpha_h$ and $\beta_h$:
\begin{equation*}
\begin{aligned}
   C_{*}=& \, C_{13}\begin{multlined}[t] \left\{((1+\sqrt{T})L\overline{h}^{s-d/2}+M(1+\overline{h}^{s-d/2})+L\overline{h}^{s-1}+m+1) \right.\\[0.5mm]
  \left.  \times \, \text{exp}(\, C_{14}\,(L^2\overline{h}^{2(s-1)}+M^2+T+1)T) +1\right \}. \end{multlined}
\end{aligned}
\end{equation*}
Therefore, from \eqref{est_fixedpoint_} and the fact that $v_h \in \mathcal{B}_h$ we infer that
\begin{equation*}
\begin{aligned}
& \|u-u_h\|_{L^\infty L^2}+\left\|u_t-u_{h,t} \right\|_{L^
\infty L^2}+\left\|u_{tt}-u_{h,tt} \right\|_{L^
2 L^2}\\[1mm]
&+h\|\nabla (u-u_h)\|_{L^\infty L^2}+ h\left \|\nabla u_{t}-u_{h,t} \right\|_{L^\infty L^2}  \\[1mm]
\leq& \, C_{*}\{3+2k L\}M h^s \leq \, L h^s,
\end{aligned}
\end{equation*}
for sufficiently small $m$, $M$, and $\overline{h}$. We can conclude that $u_h \in \mathcal{B}_h$ and, therefore, $\mathcal{F}(\mathcal{B}_h) \subset \mathcal{B}_h$.\\[2mm]
\noindent \textbf{$\mathcal{F}$ is a contraction.}  Let $v_h^{(1)}, v_h^{(2)} \in \mathcal{B}_h$ and $u^{(1)}_h=\mathcal{F} v^1_{h}$, $u^{(2)}_{h}=\mathcal{F}v^{(2)}_{h}$. We want to show that
$$\|\mathcal{F}v^{(1)}_h-\mathcal{F}v^{(1)}_h \|_{\mathcal{X}}\, \leq \, q \|v_h^{(1)}-v^{(2)}_h \|_{\mathcal{X}}, \quad 0<q<1. $$
We note that the difference $\psi_h=u^{(1)}_h-u^{(2)}_h$ satisfies
\begin{equation*}
\begin{aligned}
&\displaystyle ((1-2k v^{(1)}_{h})\psi_{h, tt}, \phi)+ c^2 (\nabla \psi_h, \nabla \phi)+b (\nabla \psi_{h, t}, \nabla \phi)-2k \left (v^{(2)}_{h,t} \psi_{h, t}, \phi \right)\\
=& \,  \left ( 2ku^{(1)}_{h,t} \left(v^{(1)}_{h,t}-v^{(2)}_{h,t}\right) +2ku^{(2)}_{h,tt} \left(v_h^{(1)}-v_h^{(2)}\right), \phi \right),
\end{aligned}
\end{equation*}
with zero initial conditions, for all $\phi \in S_h$. The equation can be seen as a special case of the PDE we considered in Theorem~\ref{thm:existence_linear} if we choose
\begin{equation*}
\begin{aligned}
& \alpha_h=1-2k v^{(1)}_{h}, \quad \beta_h=-2k \, v^{(2)}_{h,t}, \\
& f_h= 2ku^{(1)}_{h,t} \left(v^{(1)}_{h,t}-v^{(2)}_{h,t}\right) +2ku^{(2)}_{h,tt} \left(v_h^{(1)}-v_h^{(2)}\right),
\end{aligned}
\end{equation*}
and zero initial conditions. Owing to Theorem~\ref{thm:existence_linear}, we then have the a priori bound
\begin{equation} \label{contractivity}
\begin{aligned}
    & \left\|\psi_{h, tt}\right\|^2_{L^2L^2}+ \|\nabla \psi_h\|^2_{L^\infty L^2}+\left\|\nabla \psi_{h, t} \right\|^2_{ L^\infty L^2} +\left\|\nabla \psi_{h, t}\right\|^2_{L^2L^2}  \\
          \leq& \, C_{*} k^2 \|u^{(1)}_{h,t} (v^{(1)}_{h,t}-v^{(2)}_{h,t})+ u^{(2)}_{h,tt} (v_h^{(1)}-v_h^{(2)})\|^2_{L^2 L^ 2} \smallskip \\
          \leq& \, C_{*} k^2  C^2_{H_0^1, L^4}\begin{multlined}[t] \left(\|u^{(1)}_{h,t}\|^2_{L^\infty L^4}\|v^{(1)}_{h,t}-v^{(2)}_{h,t}\|^2_{L^2 H_0^1} \right.\\
    \left. +\| u^{(2)}_{h,tt}\|^2_{L^2 L^4}\| v_h^{(1)}-v_h^{(2)}\|^2_{L^\infty H_0^1}\right), \end{multlined}
\end{aligned}
\end{equation}
where the constant above is computed according to \eqref{Thm1_const},
$$C_{*}=C_{15}\, \textup{exp}( \, C_{16}\, (L^2\overline{h}^{2(s-1)}+M^2+1)T).$$
We next show that $\|u^{(1)}_{h,t}\|_{L^\infty L^4}$ and $\| u^{(2)}_{h,tt}\|_{L^2 L^4}$ in \eqref{contractivity} can be made small so that $\mathcal{F}$ is a contraction. On account of the inverse properties \eqref{inverse_estimate} of $\{S_h\}_{0<h\leq \overline{h}}$ , we first find that
\begin{equation*}
\begin{aligned}
\| u^{(2)}_{h,tt}\|_{L^2 L^4}\leq&\, \|u_{h, tt}^{(2)}-R_h u_{tt}\|_{L^2 L^4}+\| R_h u_{tt}\|_{L^2 L^4}\\
\leq&\, h^{-d/4} C_{\textup{inv}} \|u_{h, tt}^{(2)}-R_h u_{tt} \|_{L^2 L^2}+\| R_h u_{tt} \|_{L^2 L^4}\\
\leq&\, h^{-d/4} C_{\textup{inv}} \|u_{h, tt}^{(2)}-u_{tt} \|_{L^2 L^2}+h^{-d/4} C_{\textup{inv}} \|u_{tt}-R_h u_{tt} \|_{L^2 L^2}+\| R_h u_{tt} \|_{L^2 L^4}.
\end{aligned}
\end{equation*}
We can then bound $\varrho_{tt}=u_{tt}-R_h u_{tt}$ by employing Lemma~\ref{Lemma_projector_I} to get
\begin{equation} \label{contractivity_est_1}
\begin{aligned}
\| u^{(2)}_{h,tt}\|_{L^2 L^4}
\leq&\, C_{\textup{inv}} h^{-d/4} (Lh^s+C h^s\| u_{tt}\|_{L^2 H^s})+\left\| R_h u_{tt} \right\|_{L^2 L^4} \\
\leq&\,  C_{\textup{inv}}(L+CM) \overline{h}^{s-d/4}+ CM.
\end{aligned}
\end{equation}
Above, we have also made use of the fact that $\|R_h u_{tt}\|_{L^2L^4} \leq C \|u_{tt}\|_{L^2 H^s}$ for some $C>0$ independent of $h$ and $u$. Furthermore, since $u_h^{(1)} \in \mathcal{B}_h$, we can bound the term $ \|u^{(1)}_{h,t}\|_{L^\infty L^4}$ as follows
\begin{equation} \label{contractivity_est_2}
\begin{aligned}
     \|u^{(1)}_{h,t}  \|_{L^\infty L^4} \leq&\, C_{H_0^1, L^4} \left(\|u_t\|_{L^\infty H^1}+\|u_t-u^{(1)}_{h,t}\|_{L^\infty H_0^1}\right) \\
     \leq&\, C (M+L\overline{h}^{s-1}),
\end{aligned}
\end{equation}
where $C>0$ is independent of $h$ and $u$. Altogether from \eqref{contractivity}, \eqref{contractivity_est_1}, and \eqref{contractivity_est_2}, for sufficiently small $M$ and $\overline{h}$, we can conclude that $\mathcal{F}$ is contractive with respect to the topology induced by $|||\cdot|||$. The statement now follows by applying the Banach fixed-point theorem.
\end{proof}
We note that due to the presence of the strong damping in the model, the assumed regularity of the solution $u$ is to be realistically expected. The higher the sound diffusivity $b$ is, the less pronounced nonlinear effects are in the propagation, such as the steepening of the wave. We refer the interested reader to, e.g.,~\cite[Section 7.1.1]{fritz2018well} for a detailed discussion on how the $b$-damping influences the behavior of the nonlinear acoustic models.
\begin{remark} If $\Omega$ is a bounded interval in $\mathbb{R}$, we can rely on the embedding $H_0^1(\Omega) \hookrightarrow L^\infty(\Omega)$ to avoid degeneracy of the semi-discrete Westervelt equation and then we do not need to employ the inverse properties of $\{S_h\}_{0<h \leq \overline{h}}$. Indeed, the $L^\infty$ bounds \eqref{Linf_bound1_b} and \eqref{Linf_bound2} in the proof of Theorem~\ref{thm:Westervelt} can be replaced by
\begin{equation*}
\begin{aligned}
&\|v_h\|_{L^\infty L^\infty} \leq C_{H_0^1, L^\infty}(L \overline{h}^{s-1}+CM), \quad
\left \|v_{h,t} \right \|_{L^2 L^\infty} \leq C_{H_0^1, L^\infty}(L \overline{h}^{s-1}+CM), 
\end{aligned}
\end{equation*}
respectively, for $v_h \in \mathcal{B}_h$. Similarly, when proving contractivity, we can replace estimate  \eqref{contractivity} by 
\begin{equation*} 
\begin{aligned}
    & \left\|\psi_{h, tt}\right\|^2_{L^2L^2}+ \|\nabla \psi_h\|^2_{L^\infty L^2}+\left\|\nabla \psi_{h, t} \right\|^2_{ L^\infty L^2} +\|\nabla \psi_{h, t}\|^2_{L^2L^2}  \smallskip \\
          \leq& \,  C_{*} k^2  C^2_{H_0^1, L^\infty} \begin{multlined}[t]\left(\|u^{(1)}_{h,t}\|^2_{L^\infty H_0^1}\, \|v^{(1)}_{h,t}-v^{(2)}_{h,t}\|^2_{L^2 L^2} \right.\\
     \left.    +\| u^{(2)}_{h,tt}\|^2_{L^2 L^2}\, \| v_h^{(1)}-v_h^{(2)}\|^2_{L^\infty H_0^1}\right).\end{multlined}
\end{aligned}
\end{equation*}
Since $u^{(1)}_{h}$, $u^{(2)}_{h} \in \mathcal{B}_h$, we can easily show that $\mathcal{F}$ is a contraction for sufficiently small $M$ and $\overline{h}$. Therefore, the a priori bound \eqref{a_priori_bound} holds in $1$D as well with $1 < s \leq 2$. 
\end{remark}
\hspace*{1em} Although beyond the scope of the present work, we expect that the error analysis of a fully discrete scheme would rely on similar theoretical tools and an analogous fixed-point approach. We refer to~\cite[Chapter 8]{raviart1998introduction} which may serve as a first step in the direction of the error analysis for the Newmark time-stepping scheme that is often used in practice.
\section{Numerical results} \label{Section:NumExample}
To illustrate the theory, we conduct two numerical experiments using a MATLAB implementation based on the GeoPDEs package~\cite{vazquez2016new}. 
\subsection{Propagation in a channel} We first perform an experiment in a 1D channel setting. For the medium, we choose water with the parameter values
        $$c=1500 \, \textup{m}/\textup{s}, \ \rho = 1000 \, \textup{kg}/\textup{m}^3,\ b=6 \cdot 10^{-9} \, \textup{m}^2/\textup{s}, \ \beta_a=3.5;$$ cf.~\cite{manfred}.  Following~\cite[Algorithm 1]{muhr2017isogeometric}, to resolve the nonlinearities, we employ a fixed-point iteration with respect to the second time derivative. The tolerance is set to $\textup{TOL}=10^{-8}$. Time stepping is performed by employing the Newmark method~\cite{newmark1959method} with the parameters $(\beta, \gamma)=(0.45, 0.75)$. The values are chosen in this way since they were shown to provide good results in simulations of the nonlinear acoustic equations; see~\cite{fritz2018well, muhr2017isogeometric}. For all spatial refinements, we have $2001$ grid points in time with the final time set to $T=37\, \mu \textup{s}$. We conduct this experiment with the initial data 
\begin{equation*} 
\begin{aligned}
    (u_0, u_1)= \left(A_1 \, \textup{exp}\left (-\frac{(x-\mu)^2}{2\sigma_1^2}\right),\ A_2 (x-\mu) \, \textup{exp}\left (-\frac{(x-\mu)^2}{2\sigma_2^2}\right) \right), \\
\end{aligned}
\end{equation*}
where $A_1=1.2 \cdot 10^8 \, \textup{Pa}$, $A_2=-10^{11} \, \textup{Pa}$, $\sigma_1=0.015$, $\sigma_2=0.02$, and $\mu=0.1$. We note that in this setting the upper bound for the acoustic pressure that guarantees non-degeneracy amounts to $1/(2k) \approx 214 \,$MPa. \\
 \indent We use piecewise linear elements in space to compute solutions on different discretization levels. The numerical solution is computed on level $N$, $N \in [1, 6]$, by employing $100 \cdot 2^{N-1}$  elements for the channel length of $0.2 \, \textup{m}$. The reference solution is taken to be the solution on a very fine mesh, where $N=8$. After obtaining the numerical solution on a coarse mesh, we interpolate it linearly to the mesh on level $N=8$ and compare to the reference. Figure~\ref{fig:Snapshots} displays snapshots of the reference pressure wave as it propagates. \\
 \indent Let $e_N$ denote the error in a certain norm on level $N$. We determine the order of convergence on this level via
\begin{align*}
\textup{order}_{N}=\frac{log(e_{N-1}/e_{N})}{log(2)}.
\end{align*}
The numerical error orders obtained in the experiments are stated in Table~\ref{table1} and Table~\ref{table2}. We see that they agree with our theoretically predicted bounds in Theorem~\ref{thm:Westervelt}. \newpage
\begin{figure}[h!]
\begin{center}
\input{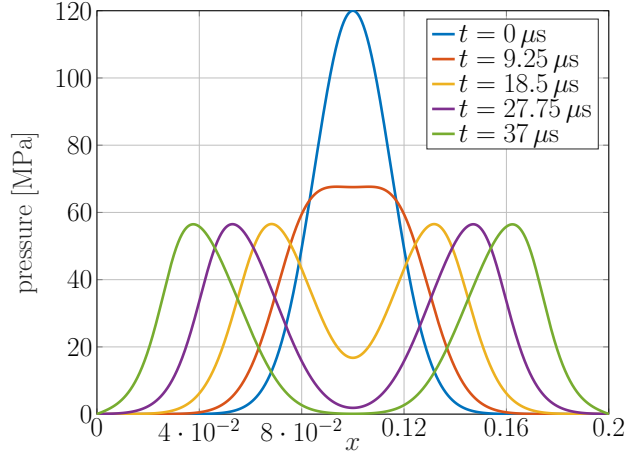}
\vspace{-1mm}
\caption{Snapshots of the reference pressure wave in a channel
\label{fig:Snapshots}} 
\end{center}
\end{figure}      
\begin{center}
 \begin{tabular}{||c |c  | c ||} 
 \hline
 level $N$ &  $\|u-u_h\|_{L^\infty L^2}$ &  $\|\nabla (u-u_h)\|_{L^\infty L^2}$  \\ [0.5ex] 
 \hline\hline
  \hline
  2 & 1.9997 & 0.9993   \\  
 \hline
 3 & 2.0011  & 1.0003   \\
 \hline
 4 & 2.0042   & 1.0021  \\
 \hline
 5 &    2.0168   & 1.0085  \\
 \hline
 6 & 2.0692 &  1.0352  \\  
 \hline 
\end{tabular}
\vspace*{-1mm}
\captionof{table}{Order of discretization errors for $u_h$.} \label{table1}
\end{center}
~\\                                          
\begin{center}
 \begin{tabular}{||c |c  | c | c ||} 
 \hline
 level $N$ &  $\left \|u_t-u_{h,t} \right\|_{L^\infty L^2}$ &  $\left \|\nabla (u_t-u_{h,t}) \right \|_{L^\infty L^2}$ &  $\left \|u_{tt}-u_{h, tt} \right \|_{L^2 L^2}$ \\ [0.5ex] 
 \hline\hline
  \hline
  2 &  2.0068 & 1.2258    & 2.0172     \\
 \hline
 3 &  2.0039   & 1.0691  & 2.0076  \\
 \hline
 4 & 2.0050 & 1.0201 & 2.0059   \\
 \hline
 5 & 2.0171 & 1.0131  & 2.0173  \\
 \hline
 6 & 2.0697 & 1.0363  & 2.0694  \\  
 \hline 
\end{tabular}
\end{center}
\vspace*{-1mm}
\captionof{table}{Order of discretization errors for $u_{h, t}$ and $u_{h, tt}$.} \label{table2}
\setlength{\parindent}{1em}
\subsection{Focused ultrasound} In our second example, we consider a more application-oriented setting of a focused-ultrasound problem. In ultrasound applications, the sound is often excited by transducers arranged on a spherical surface~\cite{cline1999ultrasound, manfred}. The wave then self-focuses as it propagates; see~Figure~\ref{fig:self-focusing}. This approach results in localized high-pressure values, and is therefore often used in non-invasive treatments of kidney stones and certain types of cancer~\cite{hill1995high, illing2005safety, kennedy2003high}.\\
\indent In the present experiment, our computational domain is a rectangle $[0, 0.04]\, \textup{m} \times [0, 0.05]\, \textup{m}$ with a curved bottom side that belongs to the circle centered at $(0.02 \, \textup{m}, 0.04 \, \textup{m})$ with radius $R^2=0.002\, \textup{m}^2$. On this bottom side we impose Neumann boundary conditions as a modulated sinusoidal wave: 
\begin{align*}
\frac{\partial u}{\partial n}= \begin{cases} g_0\, \sin(\omega t)\,(1+ \sin(wt/4)) , \ t> 2 \pi/w, \smallskip \\ g_0\, \sin(\omega t), \hspace{25mm}  t \leq 2\pi/w. \end{cases} \quad \text{on } \Gamma_{\textup{N}}.
\end{align*}
The angular frequency is taken to be $w=2 \pi f$ with $f= 60 \, \textup{kHz}$, and we set $g_0= 10^{7} \, \textup{Pa/m}$. On the rest of the domain sides, we impose linear absorbing boundary conditions, i.e. $$\frac{\partial u}{\partial n}=-\frac{1}{c}u_t \quad \text{on } \partial \Omega \setminus \Gamma_{\textup{N}}.$$
\begin{figure}[H]
\begin{center}
\begin{minipage}{0.32\textwidth}
	\includegraphics[trim = 13.5cm 0cm 0cm 2.5cm, clip, scale=0.162]{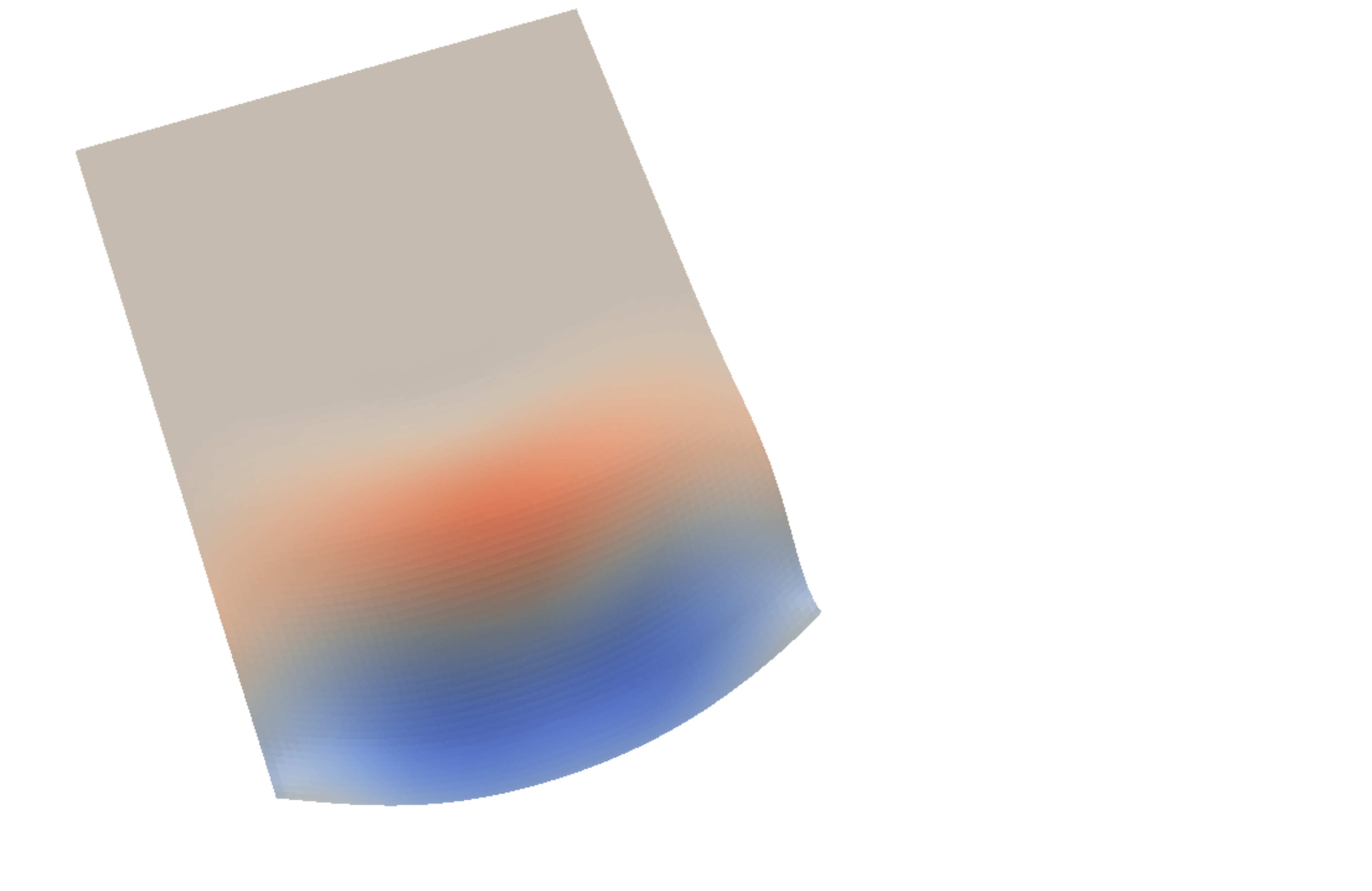}
\end{minipage}\hfill
\begin{minipage}{0.32\textwidth}
	\includegraphics[trim =13cm 0cm 0cm 2.5cm, clip, scale=0.162]{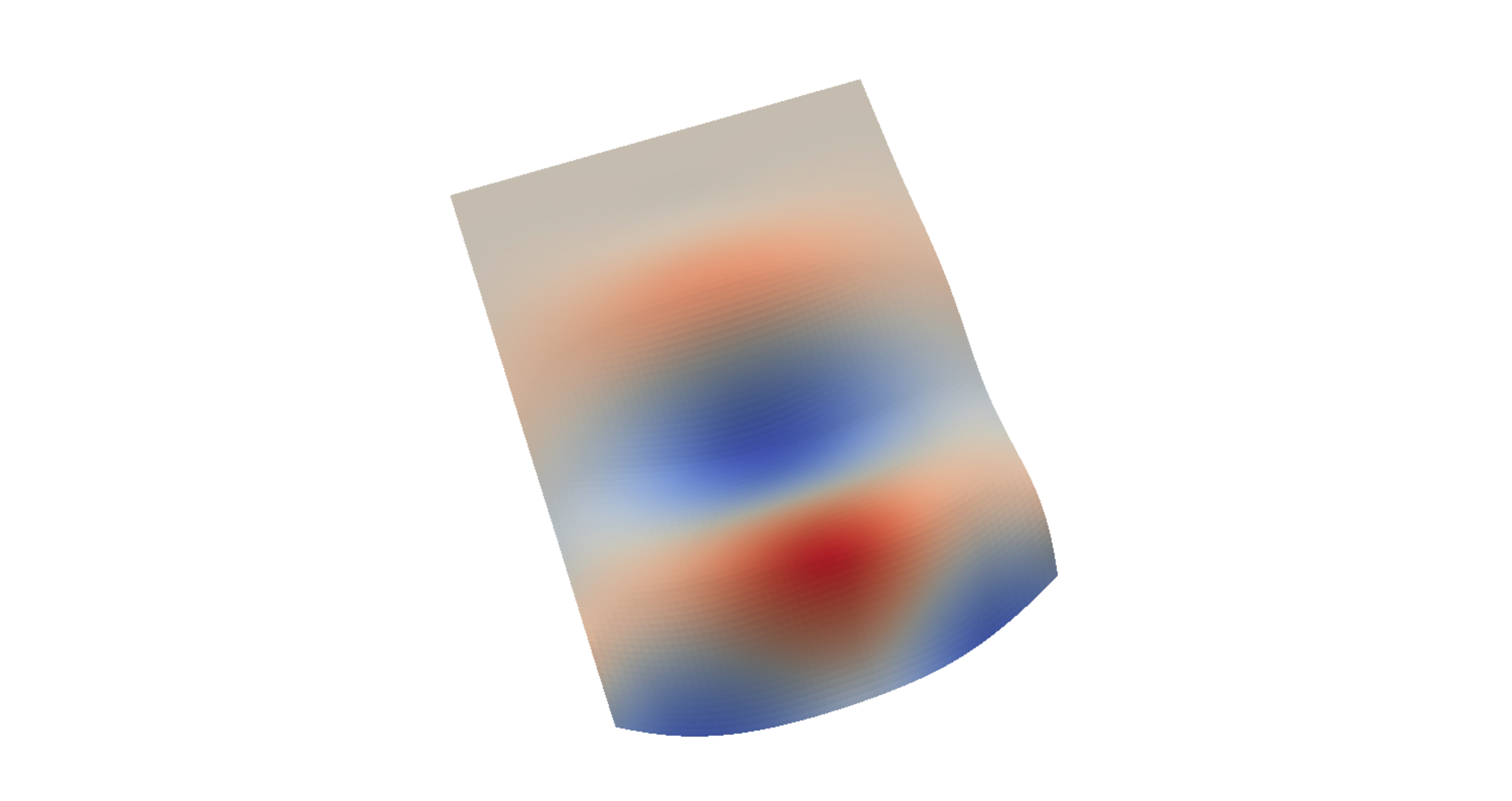}
\end{minipage}\hfill
\begin{minipage}{0.32\textwidth}
	\includegraphics[trim = 13cm 0cm 4cm 2.5cm, clip, scale=0.162]{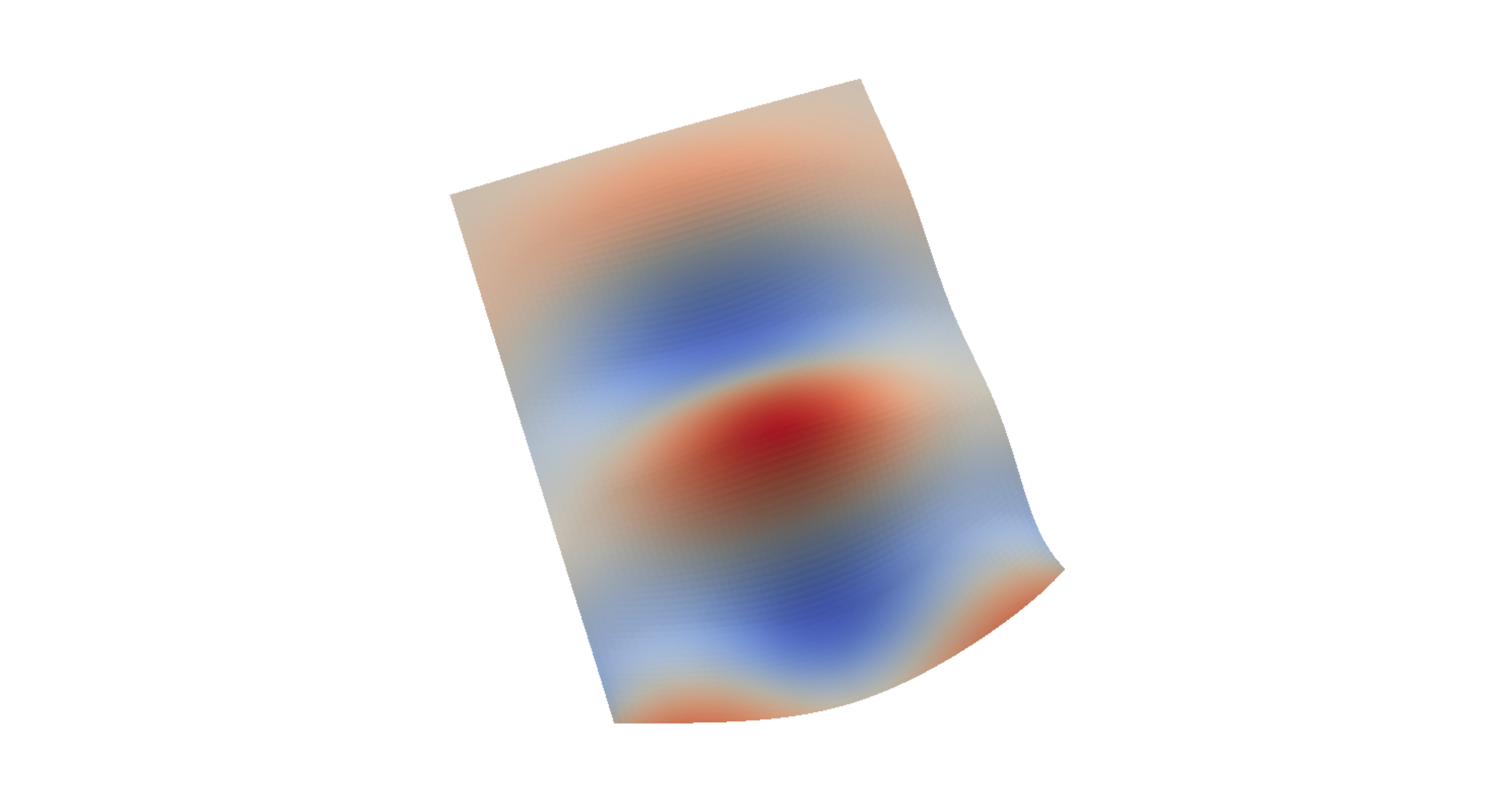}
\end{minipage}
\caption{Propagation and self-focusing of a sound wave \label{fig:self-focusing}}
\end{center}
\end{figure}
We mention that nonlinear absorbing conditions for the Westervelt equation have also been derived and investigated in~\cite{muhr2018self, shevchenko2015absorbing}.  Discretization in time is performed with $3500$ time steps for the final time $T=40 \, \mu \textup{s}$ and we again employ the Newmark scheme for time stepping with the same parameters as before. For the discretization is space, we employ a quadrilateral mesh which on the discretization level $N$ has $2^{N-1}\cdot 35$ elements in the propagation direction and $ 2^{N-1} \cdot 20$ elements in the other direction, where $N \in \{1, \ldots, 5\}$. \\
\indent On each discretization level $N$ we compute
\begin{align*}
q(u_h)=\|u^{N}_h\|_{L^\infty L^2},
\end{align*}
noting that $e_N= |q(u)-q(u^{N}_h)| \leq\, \|u-u^{N}_h\|_{L^\infty L^2}.$ Thus we have
\begin{align*}
    q(u_h) \leq \|u\|_{L^\infty L^2}+Ch^s.
\end{align*} Figure~\ref{fig:QoI} displays how the function $q$ changes with respect to $h$. The data has been then fitted to a curve $\alpha +\beta h^{\gamma}$ by employing the nonlinear least-squares solver \emph{lsqcurvefit} in MATLAB with the starting point $(1, 1 ,2)$. We obtain $\gamma \approx 1.82$ for the order of convergence. In Table~\ref{table3}, we have the orders of discretization errors for $|q(u)-q(u_h)|$ if we take the value on the highest level $N=5$ as the reference, i.e., $u=u_h^{5}$, which are again around $2$.\\[6mm] 
\begin{minipage}{0.96\textwidth}
  \begin{minipage}[b]{0.48\textwidth}
%
\begin{tikzpicture}[scale=0.42, font=\huge]

\begin{axis}[%
width=5.358in,
height=4.226in,
at={(0.899in,0.57in)},
scale only axis,
xmin=0,
xmax=0.003, xtick={0. 0002, 0.0006, 0.001, 0.0014, 0.0018, 0.0022, 0.0026, 0.003},
xlabel style={at={(0.5, -0.05)}, font=\color{white!15!black}},
xlabel={\fontsize{25}{20} $h$},
ymin=1176,
ymax=1180.5, ytick={1176.5, 1177.5, 1178.5, 1179.5, 1180.5},
ylabel style={at={(-0.13, 0.5)}, font=\color{white!15!black}},
ylabel={\fontsize{25}{20} $q(u_h) \ \, [\textup{Pa}]$},
axis background/.style={fill=white},
title style={font=\bfseries},
xmajorgrids,
ymajorgrids,
legend style={at={(0.072,0.768)}, anchor=south west, legend cell align=left, align=left, draw=white!15!black}
]
\addplot [color=black, draw=none, mark=o, mark options={solid, black, scale=3}]
  table[row sep=crcr]{%
0.00250000000005457	1180.47875053766\\
0.00125000000002728	1177.48279070331\\
0.000624999999899956	1176.60715459484\\
0.000312500000063665	1176.37817570144\\
0.000156249999918145	1176.32006256018\\
};
\addlegendentry{Data}

\addplot [color=blue, line width=3.2pt]
  table[row sep=crcr]{%
0.00250000000005457	1180.47981402405\\
0.00242897727275704	1180.26579494478\\
0.00235795454545951	1180.0568244548\\
0.00228693181816197	1179.85293014747\\
0.00223958333322116	1179.71983478865\\
0.00219223484850772	1179.58901693225\\
0.00214488636356691	1179.46048553386\\
0.00209753787885347	1179.33424978104\\
0.00205018939391266	1179.21031910463\\
0.00200284090919922	1179.08870319095\\
0.00195549242425841	1178.96941199484\\
0.0019081439393176	1178.85245575371\\
0.00186079545460416	1178.73784500278\\
0.00181344696966335	1178.62559059135\\
0.00176609848494991	1178.51570370063\\
0.00171875000000909	1178.40819586291\\
0.00167140151506828	1178.30307898247\\
0.00162405303035484	1178.20036535837\\
0.00157670454541403	1178.10006770924\\
0.00152935606070059	1178.00219920055\\
0.00148200757575978	1177.90677347441\\
0.00143465909081897	1177.81380468242\\
0.00138731060610553	1177.72330752197\\
0.00133996212116472	1177.63529727635\\
0.00129261363645128	1177.54978985937\\
0.00124526515151047	1177.46680186507\\
0.00119791666656965	1177.38635062346\\
0.00115056818185622	1177.30845426302\\
0.0011032196969154	1177.2331317814\\
0.00105587121220196	1177.16040312555\\
0.00100852272726115	1177.09028928317\\
0.00096117424232034	1177.02281238751\\
0.000913825757606901	1176.9579958384\\
0.000866477272666089	1176.89586444284\\
0.00081912878795265	1176.83644457956\\
0.000771780303011838	1176.77976439328\\
0.000748106060655118	1176.75246101317\\
0.000724431818071025	1176.72585402614\\
0.000700757575714306	1176.69994755766\\
0.000677083333357587	1176.67474589606\\
0.000653409091000867	1176.65025350474\\
0.000629734848416774	1176.62647503583\\
0.000606060606060055	1176.60341534551\\
0.000582386363703336	1176.58107951111\\
0.000558712121119243	1176.55947285046\\
0.000535037878762523	1176.53860094381\\
0.000511363636405804	1176.51846965876\\
0.000487689394049085	1176.49908517878\\
0.000464015151464992	1176.48045403616\\
0.000440340909108272	1176.46258315007\\
0.000416666666751553	1176.44547987116\\
0.00039299242416746	1176.42915203396\\
0.00036931818181074	1176.41360801915\\
0.000345643939454021	1176.39885682829\\
0.000321969696869928	1176.38490817443\\
0.000298295454513209	1176.3717725934\\
0.000274621212156489	1176.35946158249\\
0.00025094696979977	1176.34798777618\\
0.000227272727215677	1176.33736517294\\
0.000203598484858958	1176.32760943492\\
0.000179924242502238	1176.31873829458\\
0.000156249999918145	1176.31077212523\\
};
\addlegendentry{Fitted curve}

\end{axis}
\end{tikzpicture}%
\captionsetup{justification=raggedleft,singlelinecheck=false}

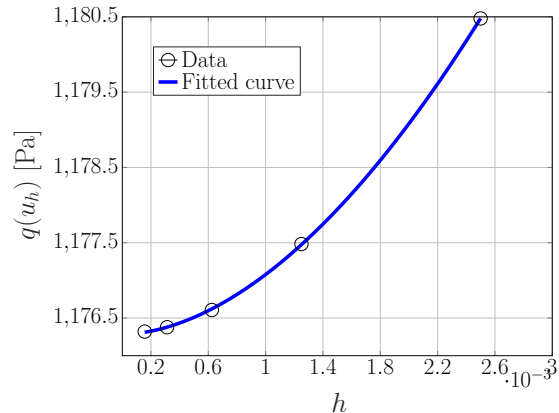
\captionof{figure}{$q(u_h)$ for different $h$
\label{fig:QoI}}
  \end{minipage}
 \hspace*{9mm}
  \begin{minipage}[b]{0.48\textwidth}
\begin{center}
 \begin{tabular}{||c |c ||} 
 \hline
 level $N$ & $\, |q(u)-q(u^{N}_h)| \,$   \\ [0.5ex] 
 \hline\hline
  \hline
  2 & 1.8386   \\  
 \hline
 3 &   2.0179    \\
 \hline
 4 & 2.3046     \\
 \hline
\end{tabular}
\end{center}
\vspace{7.5mm}
\captionof{table}{Order of error $e_{N}$.} \label{table3}
    \end{minipage}
  \end{minipage}
  
\section*{Acknowledgments}
We would like to thank Markus Muhr for helpful discussions about the numerical simulation of the problem. The funds provided by the Deutsche Forschungsgemeinschaft under the grant number WO 671/11-1 are gratefully acknowledged.

\bibliographystyle{siamplain}
\bibliography{references}

\begin{thebibliography}{10}

\bibitem{baker1976error}
{\sc G.~A. Baker}, {\em Error estimates for finite element methods for second
  order hyperbolic equations}, SIAM {J}ournal on {N}umerical {A}nalysis, 13
  (1976), pp.~564--576.

\bibitem{bales1994continuous}
{\sc L.~Bales and I.~Lasiecka}, {\em Continuous finite elements in space and
  time for the nonhomogeneous wave equation}, Computers \& Mathematics with
  Applications, 27 (1994), pp.~91--102.

\bibitem{bangerth1999finite}
{\sc W.~Bangerth and R.~Rannacher}, {\em Finite element approximation of the
  acoustic wave equation: {E}rror control and mesh adaptation}, East West
  Journal of Numerical Mathematics, 7 (1999), pp.~263--282.

\bibitem{bjorno2002forty}
{\sc L.~Bj{\o}rn{\o}}, {\em Forty years of nonlinear ultrasound}, Ultrasonics,
  40 (2002), pp.~11--17.

\bibitem{brenner2007mathematical}
{\sc S.~Brenner and R.~Scott}, {\em The mathematical theory of finite element
  methods}, vol.~15, Springer Science \& Business Media, 2007.

\bibitem{cavalcanti2002existence}
{\sc M.~Cavalcanti, V.~Cavalcanti, P.~JS~Filho, and J.~Soriano}, {\em Existence
  and uniform decay of solutions of a parabolic-hyperbolic equation with
  nonlinear boundary damping and boundary source term}, Communications in
  Analysis and Geometry, 10 (2002), pp.~451--466.

\bibitem{clason2009boundary}
{\sc C.~Clason, B.~Kaltenbacher, and S.~Veljovi{\'c}}, {\em Boundary optimal
  control of the {W}estervelt and the {K}uznetsov equations}, Journal of
  Mathematical Analysis and Applications, 356 (2009), pp.~738--751.

\bibitem{cline1999ultrasound}
{\sc H.~E. Cline, R.~D. Watkins, G.~R. Russell, and K.~H. Hynynen}, {\em
  Ultrasound transducer with focused ultrasound refraction plate}, (1999).
\newblock US Patent 5,873,845.

\bibitem{crighton1979model}
{\sc D.~G. Crighton}, {\em Model equations of nonlinear acoustics}, Annual
  Review of Fluid Mechanics, 11 (1979), pp.~11--33.

\bibitem{dendy1977galerkin}
{\sc J.~Dendy, Jr}, {\em Galerkin’s method for some highly nonlinear
  problems}, SIAM {J}ournal on {N}umerical {A}nalysis, 14 (1977), pp.~327--347.

\bibitem{dupont1973}
{\sc T.~Dupont}, {\em ${L}^{2}$-estimates for {G}alerkin methods for second
  order hyperbolic equations}, SIAM journal on numerical analysis, 10 (1973),
  pp.~880--889.

\bibitem{emmrich2015full}
{\sc E.~Emmrich, D.~{\v{S}}i{\v{s}}ka, and M.~Thalhammer}, {\em On a full
  discretisation for nonlinear second-order evolution equations with monotone
  damping: construction, convergence, and error estimates}, Foundations of
  Computational Mathematics, 15 (2015), pp.~1653--1701.

\bibitem{enflo2006theory}
{\sc B.~O. Enflo and C.~M. Hedberg}, {\em Theory of nonlinear acoustics in
  fluids}, vol.~67, Springer Science \& Business Media, 2006.

\bibitem{Evans}
{\sc L.~C. Evans}, {\em Partial differential equations}, vol.~19 of Graduate
  Studies in Mathematics, American Mathematical Society, 2010.

\bibitem{fierro2015nonlinear}
{\sc G.~M. Fierro, F.~Ciampa, D.~Ginzburg, E.~Onder, and M.~Meo}, {\em
  Nonlinear ultrasound modelling and validation of fatigue damage}, Journal of
  {S}ound and {V}ibration, 343 (2015), pp.~121--130.

\bibitem{fritz2018well}
{\sc M.~Fritz, V.~Nikoli{\'c}, and B.~Wohlmuth}, {\em Well-posedness and
  numerical treatment of the {B}lackstock equation in nonlinear acoustics},
  Mathematical Models and Methods in Applied Sciences, 28 (2018),
  pp.~2557--2597.

\bibitem{garcke2017well}
{\sc H.~Garcke and K.~F. Lam}, {\em Well-posedness of a {C}ahn--{H}illiard
  system modelling tumour growth with chemotaxis and active transport},
  European Journal of Applied Mathematics, 28 (2017), pp.~284--316.

\bibitem{georgoulis2013posteriori}
{\sc E.~H. Georgoulis, O.~Lakkis, and C.~Makridakis}, {\em A posteriori
  ${L}^\infty({L}^2)$-error bounds for finite element approximations to the
  wave equation}, IMA Journal of Numerical Analysis, 33 (2013), pp.~1245--1264.

\bibitem{girault2012finite}
{\sc V.~Girault and P.-A. Raviart}, {\em Finite element methods for
  {N}avier--{S}tokes equations: theory and algorithms}, vol.~5, Springer
  Science \& Business Media, 2012.

\bibitem{hamilton1998nonlinear}
{\sc M.~F. Hamilton and D.~T. Blackstock}, {\em Nonlinear acoustics}, vol.~1,
  Academic press San Diego, 1998.

\bibitem{hill1995high}
{\sc C.~Hill and G.~Ter~Haar}, {\em High intensity focused
  ultrasound--potential for cancer treatment}, The British journal of
  radiology, 68 (1995), pp.~1296--1303.

\bibitem{hoffelner2001finite}
{\sc J.~Hoffelner, H.~Landes, M.~Kaltenbacher, and R.~Lerch}, {\em Finite
  element simulation of nonlinear wave propagation in thermoviscous fluids
  including dissipation}, IEEE transactions on ultrasonics, ferroelectrics, and
  frequency control, 48 (2001), pp.~779--786.

\bibitem{illing2005safety}
{\sc R.~Illing, J.~Kennedy, F.~Wu, G.~Ter~Haar, A.~Protheroe, P.~Friend,
  F.~Gleeson, D.~Cranston, R.~Phillips, and M.~Middleton}, {\em The safety and
  feasibility of extracorporeal high-intensity focused ultrasound ({HIFU}) for
  the treatment of liver and kidney tumours in a {W}estern population}, British
  journal of cancer, 93 (2005), p.~890.

\bibitem{kagawa1992finite}
{\sc Y.~Kagawa, T.~Tsuchiya, T.~Yamabuchi, H.~Kawabe, and T.~Fujii}, {\em
  Finite element simulation of non-linear sound wave propagation}, Journal of
  Sound and Vibration, 154 (1992), pp.~125--145.

\bibitem{KaltenbacherLasiecka_Westervelt}
{\sc B.~Kaltenbacher and I.~Lasiecka}, {\em Global existence and exponential
  decay rates for the {W}estervelt equation}, Discrete and Continuous Dynamical
  Systems Series S, 2 (2009), pp.~503--523.

\bibitem{KaltenbacherLasieckaVeljovic}
{\sc B.~Kaltenbacher, I.~Lasiecka, and S.~Veljovi{\'c}}, {\em Well-posedness
  and exponential decay for the {W}estervelt equation with inhomogeneous
  {D}irichlet boundary data}, in Parabolic problems, Springer, 2011,
  pp.~357--387.

\bibitem{kaltenbacher2014efficient}
{\sc B.~Kaltenbacher, V.~Nikoli\'c, and M.~Thalhammer}, {\em Efficient time
  integration methods based on operator splitting and application to the
  {W}estervelt equation}, IMA Journal of Numerical Analysis, 35 (2014),
  pp.~1092--1124.

\bibitem{kaltenbacher2016shape}
{\sc B.~Kaltenbacher and G.~Peichl}, {\em The shape derivative for an
  optimization problem in lithotripsy}, Evolution Equations \& Control Theory,
  5 (2016).

\bibitem{manfred}
{\sc M.~Kaltenbacher}, {\em Numerical simulation of mechatronic sensors and
  actuators}, Springer, 2015.

\bibitem{kawashima1992global}
{\sc S.~Kawashima and Y.~Shibata}, {\em Global existence and exponential
  stability of small solutions to nonlinear viscoelasticity}, Communications in
  mathematical physics, 148 (1992), pp.~189--208.

\bibitem{kennedy2003high}
{\sc J.~Kennedy, G.~Ter~Haar, and D.~Cranston}, {\em High intensity focused
  ultrasound: surgery of the future?}, The British journal of radiology, 76
  (2003), pp.~590--599.

\bibitem{larsson1991finite}
{\sc S.~Larsson, V.~Thom{\'e}e, and L.~B. Wahlbin}, {\em Finite-element methods
  for a strongly damped wave equation}, IMA Journal of Numerical Analysis, 11
  (1991), pp.~115--142.

\bibitem{makridakis1993finite}
{\sc C.~G. Makridakis}, {\em Finite element approximations of nonlinear elastic
  waves}, Mathematics of Computation, 61 (1993), pp.~569--594.

\bibitem{meyer2011optimal}
{\sc S.~Meyer and M.~Wilke}, {\em Optimal regularity and long-time behavior of
  solutions for the {W}estervelt equation}, Applied Mathematics \&
  Optimization, 64 (2011), pp.~257--271.

\bibitem{muhr2018self}
{\sc M.~Muhr, V.~Nikoli{\'c}, and B.~Wohlmuth}, {\em Self-adaptive absorbing
  boundary conditions for quasilinear acoustic wave propagation}, Journal of
  Computational Physics, 388 (2019), pp.~279--299.

\bibitem{muhr2017isogeometric}
{\sc M.~Muhr, V.~Nikoli{\'c}, B.~Wohlmuth, and L.~Wunderlich}, {\em
  Isogeometric shape optimization for nonlinear ultrasound focusing}, Evolution
  Equations \& Control Theory, 8 (2019), pp.~163--202.

\bibitem{muller2008nonlinear}
{\sc M.~Muller, D.~Mitton, M.~Talmant, P.~Johnson, and P.~Laugier}, {\em
  Nonlinear ultrasound can detect accumulated damage in human bone}, Journal of
  Biomechanics, 41 (2008), pp.~1062--1068.

\bibitem{newmark1959method}
{\sc N.~M. Newmark}, {\em A method of computation for structural dynamics},
  Journal of the engineering mechanics division, 85 (1959), pp.~67--94.

\bibitem{novell2009exploitation}
{\sc A.~Novell, M.~Legros, N.~Felix, and A.~Bouakaz}, {\em Exploitation of
  capacitive micromachined transducers for nonlinear ultrasound imaging}, IEEE
  Transactions on Ultrasonics, Ferroelectrics, and Frequency Control, 56
  (2009), pp.~2733--2743.

\bibitem{ortner2007discontinuous}
{\sc C.~Ortner and E.~S{\"u}li}, {\em Discontinuous {G}alerkin finite element
  approximation of nonlinear second-order elliptic and hyperbolic systems},
  SIAM Journal on Numerical Analysis, 45 (2007), pp.~1370--1397.

\bibitem{pinton2011effects}
{\sc G.~Pinton, J.-F. Aubry, M.~Fink, and M.~Tanter}, {\em Effects of nonlinear
  ultrasound propagation on high intensity brain therapy}, Medical physics, 38
  (2011), pp.~1207--1216.

\bibitem{raviart1998introduction}
{\sc P.-A. Raviart, J.-M. Thomas, P.~G. Ciarlet, and J.~L. Lions}, {\em
  Introduction {\`a} l'analyse num{\'e}rique des {\'e}quations aux
  d{\'e}riv{\'e}es partielles}, vol.~2, Dunod Paris, 1998.

\bibitem{roubivcek2013nonlinear}
{\sc T.~Roub{\'\i}{\v{c}}ek}, {\em Nonlinear partial differential equations
  with applications}, vol.~153, Springer Science \& Business Media, 2013.

\bibitem{scott1990finite}
{\sc L.~R. Scott and S.~Zhang}, {\em Finite element interpolation of nonsmooth
  functions satisfying boundary conditions}, Mathematics of Computation, 54
  (1990), pp.~483--493.

\bibitem{shevchenko2015absorbing}
{\sc I.~Shevchenko and B.~Kaltenbacher}, {\em Absorbing boundary conditions for
  nonlinear acoustics: {T}he {W}estervelt equation}, Journal of Computational
  Physics, 302 (2015), pp.~200--221.

\bibitem{sinha2003effect}
{\sc R.~K. Sinha, A.~K. Pani, and S.~K. Chung}, {\em The effect of spatial
  quadrature on the semidiscrete finite element {G}alerkin method for a
  strongly damped wave equation}, Numerical Functional Analysis and
  Optimization, 24 (2003).

\bibitem{suli2000priori}
{\sc E.~S\"uli and C.~Wilkins}, {\em A priori analysis for the semi-discrete
  approximation to the nonlinear damped wave equation}, tech. report, Technical
  Report no. 99/09, Oxford University Computing Laboratory, 2000.

\bibitem{thomee1984galerkin}
{\sc V.~Thom{\'e}e}, {\em Galerkin finite element methods for parabolic
  problems}, vol.~1054, Springer, 1984.

\bibitem{thomee2004maximum}
{\sc V.~Thom{\'e}e and L.~Wahlbin}, {\em Maximum-norm estimates for
  finite-element methods for a strongly damped wave equation}, BIT Numerical
  Mathematics, 44 (2004), pp.~165--179.

\bibitem{tsuchiya2003finite}
{\sc T.~Tsuchiya, Y.~Kagawa, M.~Doi, and T.~Tsuji}, {\em Finite element
  simulation of non-linear acoustic generation in a horn loudspeaker}, Journal
  of sound and vibration, 266 (2003), pp.~993--1008.

\bibitem{vazquez2016new}
{\sc R.~V{\'a}zquez}, {\em A new design for the implementation of isogeometric
  analysis in {O}ctave and {M}atlab: {G}eo{PDE}s 3.0}, Computers \& Mathematics
  with Applications, 72 (2016), pp.~523--554.

\bibitem{walsh2007finite}
{\sc T.~Walsh and M.~Torres}, {\em Finite element methods for nonlinear
  acoustics in fluids}, Journal of Computational Acoustics, 15 (2007),
  pp.~353--375.

\bibitem{westervelt1963parametric}
{\sc P.~J. Westervelt}, {\em Parametric acoustic array}, The Journal of the
  Acoustical Society of America, 35 (1963), pp.~535--537.

\end{thebibliography}
\end{document}